\documentclass[10pt,a4paper]{amsart}
\usepackage[english]{babel}
\usepackage{amssymb}
\usepackage[left=2.5cm,right=2.5cm,top=2.5cm,bottom=2.5cm]{geometry}
\usepackage{epsfig}
\usepackage{url}

\usepackage[displaymath,mathlines,pagewise]{lineno}

\title[Limit models in heterogeneous catalysis]{The fast-sorption--fast-surface-reaction limit of a heterogeneous catalysis model}
\author{Bj\"orn Augner, Dieter Bothe}
\address{Technische Universit\"at Darmstadt, Fachbereich Mathematik, Schlossgartenstra\ss{}e 7, 64289 Darmstadt.}
\email{augner@mma.tu-darmstadt.de, bothe@mma.tu-darmstadt.de}
 
 \newtheorem{theorem}{Theorem}[section]
 \newtheorem{remark}[theorem]{Remark}
 \newtheorem{proposition}[theorem]{Proposition}
 
 \newtheorem{lemma}[theorem]{Lemma}

 \newtheorem{assumption}[theorem]{Assumption}
 \newtheorem{example}[theorem]{Example}
 \newtheorem{principle}[theorem]{Principle}
 \newtheorem{model}[theorem]{Model}

 \renewcommand{\vec}{\bd}
 
 \DeclareMathOperator{\R}{\mathbb{R}}
 \DeclareMathOperator{\N}{\mathbb{N}}

 \DeclareMathOperator{\dv}{\operatorname{div}}

 \DeclareMathOperator{\lin}{\mathrm{span}}
 
 \renewcommand{\ker}{\mathrm{N}}
 
 \newcommand{\bd}[1]{\boldsymbol{#1}}
 \newcommand{\bb}[1]{\boldsymbol{#1}}

 \newcommand{\dd}{\, \mathrm{d}}
 \newcommand{\norm}[1]{\left\| #1 \right\|} 
  
 \newcommand{\abs}[1]{\left| #1 \right|}
 
 \newcommand{\diag}{\operatorname{diag}}

 \newcommand{\BB}{\mathcal{B}}
 \newcommand{\LL}{\mathcal{L}}
 \newcommand{\WW}{\mathcal{W}}
 \newcommand{\CC}{\mathcal{C}}
 \newcommand{\DD}{\mathcal{D}}
 \newcommand{\II}{\mathcal{I}}

\begin{document}
 \allowdisplaybreaks[1]

 \begin{abstract}
  Every mathematical model describing physical phenomena is an approximation to model reality, hence has its limitations. Depending on characteristic values of the variables in the model, different aspects of the model and, e.g., thermodynamic mechanisms have to be emphasised, or may be neglected in a reduced limit model.
  Within this paper, a heterogeneous catalysis system will be considered consisting of a bulk phase $\Omega$ (chemical reactor) and an active surface $\Sigma = \partial \Omega$ (catalytic surface), between which chemical substances are exchanged via adsorption (transport of mass from the bulk boundary layer adjacent to the surface, leading to surface-accumulation by a transformation into an adsorbed form) and desorption (vice versa). Quite typically, as is the purpose of catalysis, chemical reactions on the surface occur several orders of magnitude faster than, say, chemical reactions within the bulk, and sorption processes are often quite fast as well.
  Starting from the non-dimensional version, different limit models, especially for fast surface chemistry and fast sorption at the surface, are considered. For a particular model problem, questions of local-in-time existence of strong and classical solutions, positivity of solutions and blow-up criteria for global existence are addressed.
 \end{abstract} 
 

 
 \maketitle

 \vspace{0.5cm}
 Version of \today.
 
 \section{Introduction}
 \subsection{Continuum thermodynamic modelling of reactive fluid mixtures}
 In this paper, the modelling and analysis of a prototypical chemical reactor with a catalytic surface are considered. Mathematically, the chemical reactor is described by a bounded domain $\Omega \subseteq \R^3$, in which chemical substances $A_i$, $i = 1, \ldots, N$, may diffuse and react with each other. The evolution in time of these substances can be described on different scales of accuracy, depending on the particular interest one has in the model:
 These include, e.g.\ atomistic models, models derived from statistical physics, continuum thermodynamic, or effective (integral) models.
 On atomic level, for example, for each molecule of any species it is kept account of all its mechanical variables, i.e.\ position, (linear) momentum, angular momentum etc.\ and the interaction with all the other molecules has to be included in the model. Clearly, for a (numerically) computable model of the dynamics this is barely a feasible approach as the number of unknown exceeds present computing capacities. The other extreme would be the model of a perfectly stirred chemical reactor, where any spatial dependence is neglected, i.e.\ at any time it is assumed that the molecules are homogeneously distributed in the chemical reactor. Note that this approach is actually not compatible with the model of molecules, but needs an underlying \emph{thermodynamic description} of the reaction diffusion system: Instead of a molecular viewpoint, where the number $\dd n_i$ of molecules of type $A_i$ in some volume element $\dd \Omega$ could simply be counted, the concentration $c_i \approx \frac{\dd n_i}{|\dd \Omega|}$ is introduced and used as a thermodynamic variable, describing the system on a macroscopic level. In systems which are not ideally mixed, the concentration does not only depend on the time variable $t$ but on the spatial position $\vec z \in \Omega$ as well, i.e.\ $c_i = c_i(t, \vec z)$, and this is the viewpoint adapted throughout the present manuscript.
 The evolution of the thermodynamic variable $c_i$ will be described by balance laws which together with thermodynamic closure relations give a set of partial differential equations, then.
 The fundamental balance law is the \emph{individual mass balance}
  \[
   \partial_t \rho_i + \dv (\rho_i \vec v_i)
    = R_i(\vec \rho),
    \quad
    t \geq 0, \, \vec z \in \Omega
  \]
 where $\rho_i = M_i c_i$ is the mass density of species $A_i$ with $M_i > 0$ its molar mass, $\rho_i \vec v_i$ is the \emph{individual (total) mass flux} of species $A_i$ as the product of its mass density $\rho_i$ and continuum mechanical velocity $\vec v_i$, and $R_i(\vec \rho)$ the total reaction rate for species $A_i$, $i = 1, \ldots, N$. The latter two have to be modelled via appropriate closure relations and may, in general, depend on the full vector $\vec \rho = (\rho_1, \ldots, \rho_N)^\mathsf{T}$ of mass densities and its gradient $\nabla \vec \rho$.
 For the chemical reactions, it is usually assumed that they are mass-conserving (on a non-relativistic level), so that $\sum_{i=1}^N R_i(\vec \rho) = 0$, leading to the continuity equation
  \[
   \partial_t \rho + \dv (\rho \vec v)
    = 0,
    \quad
    t \geq 0, \, \vec z \in \Omega
  \] 
 with the total mass flux density $\rho \vec v = \sum_{i=1}^N \rho_i v_i$. One defines the \emph{total mass density} as $\rho = \sum_{i=1}^N \rho_i$ and the \emph{barycentric velocity} of the fluid mixture as $\vec v = \frac{1}{\rho} \sum_{i=1}^N \rho_i \vec v_i$ (assuming $\rho > 0$, i.e.\ no local vacuum) and the individual mass fluxes decompose into $\rho \vec v_i = \rho_i (\vec v + \vec u_i) = \rho_i \vec v + \vec J_i$, where $\vec u_i$ denotes the velocity of species $A_i$ and $\vec J_i = \rho_i \vec u_i$ the \emph{diffusive flux} relative to the motion of the barycentre. Note that then, by definition, $\sum_{i=1}^N \vec J_i = 0$.
 Implicitly, here, the barycentric velocity has been identified with the mass averaged velocity $\vec v^\mathrm{mass}$, which was questioned rather recently, cf.\ e.g.\ \cite{Bre04}.
 Introducing the \emph{molar reaction rates} $r_i = \frac{1}{M_i} R_i$, the balance equations for the molar concentrations, viz.\
  \[
   \partial_t c_i + \dv (c_i \vec v_i)
    = r_i(\vec c),
    \quad
    t \geq 0, \, \vec z \in \Omega
  \]
 follow directly from the individual mass balances.
 Similar to the barycentric variables, one may introduce the \emph{molar averaged velocity} $\vec v^\mathrm{mol} = \frac{1}{c} \sum_{i=1}^N c_i \vec v_i$, where $c = \sum_{i=1}^N c_i$ denotes the total concentration density of the fluid, and $\vec u_i^\mathrm{mol} = \vec v_i - \vec v^\mathrm{mol}$ (the velocity of species $A_i$ relative to the molar averaged velocity $\vec v^\mathrm{mol}$), and $\vec j_i^\mathrm{mol} = c_i \vec u_i^\mathrm{mol}$ (the \emph{molar diffusive flux}), so that the molar balance equations may be rephrased as
  \[
   \partial_t c_i + \dv (c_i \vec v^\mathrm{mol} + \vec j_i^\mathrm{mol})
    = r_i(\vec c),
    \quad
    t \geq 0, \, \vec z \in \Omega.
  \]
 Let us note in passing that diffusive mass fluxes will be denoted by capital $\vec J_i$, whereas molar diffusive fluxes will be written as $\vec j_i$.
 Without superscript, diffusive fluxes are taken relative to the barycentric motion of the fluid, whereas a superscript like $\mathrm{mol}$ indicates that the diffusive flux is taken relative to the molar averaged velocity.
 \newline
 To get a mathematically well-posed model, certain boundary conditions have to be set at the boundary $\partial \Omega$ of the chemical reactor. Typical boundary conditions considered in the mathematical literature are, e.g., Dirichlet boundary conditions, Robin boundary conditions or Neumann boundary conditions.
 Often one finds the special case of \emph{no flux} $(c_i \vec v + \vec j_i) \cdot \vec n = 0$ boundary conditions (homogeneous Neumann boundary conditions), where $\vec n$ denotes the \emph{outer normal} vector field on $\partial \Omega$.
 However, in the case of a bulk-surface reaction-diffusion system as considered here, the boundary conditions on $c_i$ and/or $\vec j_i$ at $\partial \Omega$ are rather \emph{transmission conditions} with a reaction-diffusion system on the active surface $\Sigma \subseteq \partial \Omega$.
 Here, for simplicity, the case $\Sigma = \partial \Omega$ will be considered, but a large amount of the modelling and analysis carries over to the case of a boundary $\partial \Omega$ which is disjointly decomposed into an active surface part $\Sigma \subseteq \partial \Omega$ and, say, a no-flux boundary part $\Sigma \setminus \partial \Omega$, typically both being relatively open subsets of $\partial \Omega$.
 On the surface, the \emph{surface molar concentrations} $c_i^\Sigma$ as thermodynamic variables describe the molar concentration of species $A_i^\Sigma$ (which is interpreted as an adsorbed version of species $A_i$) per area element, and it obeys the general reaction-diffusion-adsorption balance equation
  \[
   \partial_t c_i^\Sigma + \dv_\Sigma (c_i^\Sigma \vec v^{\Sigma,\mathrm{site}} + \vec j_i^{\Sigma,\mathrm{site}})
    = r_i^\Sigma(\vec c^\Sigma) + s_i^\Sigma(\vec c|_\Sigma, \vec c^\Sigma),
    \quad
    t \geq 0, \, \vec z \in \Sigma,
  \]
 where -- similar to before -- the individual mass flux is decomposed into a convective and a diffusive part via $c_i \vec v_i^\Sigma = c_i^\Sigma \vec v^{\Sigma,\mathrm{site}} + \vec j_i^{\Sigma,\mathrm{site}}$
 with $\vec v^{\Sigma,\mathrm{site}} = \frac{1}{c_S^\Sigma} \sum_{i=0}^N c_i^\Sigma \vec v_i^\Sigma$ the site averaged velocity on the surface.
 Here and in the following, we restrict ourselves to the situation where there is a maximal capacity $c_S^\Sigma$ (concentration of sites per area element on the solid surface), and each site can either be occupied by exactly one adsorbate or be free, i.e.\ unoccupied.
 In the latter case we think of the site being occupied by a species $A_0^\Sigma$ (vacancies / free sites) where we set the vacancy concentration to
  \[
   c_0^\Sigma
    = c_S^\Sigma - \sum_{i=1}^N c_i^\Sigma.
  \]
 Then $\vec j_i^\mathrm{site}(t,\cdot) = c_i \vec u_i^\mathrm{site} = c_i (\vec v_i^\Sigma - \vec v^{\Sigma,\mathrm{site}})$ denote the \emph{diffusive surface fluxes}, $r_i^\Sigma(\vec c^\Sigma)$ the molar surface reaction rates and $s_i^\Sigma(\vec c|_\Sigma, \vec c^\Sigma) = \vec j_i \cdot \vec n$ the normal flux through $\Sigma$ has the meaning of a \emph{sorption rate}. Here, by $\dv_\Sigma$ we denote the surface divergence which may be equivalently defined by charts for the (sufficiently regular) boundary $\partial \Omega$ or as $\dv_\Sigma \vec v^{\Sigma,\mathrm{mol}} = \operatorname{tr} ((\bb I - \vec n \otimes \vec n) \nabla \vec v^{\Sigma,\mathrm{mol}})$ for any continuously differentiable extension of $\vec v^{\Sigma,\mathrm{mol}}: \Sigma \rightarrow \R^n$ to a neighbourhood of $\Sigma$.
 Note that in this bulk-surface model, the molar surface concentrations $c_i^\Sigma$ do \emph{not} coincide with the restriction of the molar concentrations in the bulk phase to the surface, i.e.\ usually $c_i^\Sigma \neq c_i|_\Sigma$.
 Initially, the sorption rates $s_i^\Sigma(\vec c|_\Sigma, \vec c^\Sigma)$ are just defined as the normal fluxes.
 However, they may also be considered as being suitable \emph{models} for the sorption mechanisms at the surface so that, instead of the model
  \begin{align*}
   \partial_t c_i + \dv (c_i \vec v + \vec j_i)
    &= r_i(\vec c),
    &i = 1, \ldots, N, \, t \geq 0, \, \vec z \in \Omega,
    \\
   \partial_t c_i^\Sigma + \dv_\Sigma (c_i^\Sigma \vec v^{\Sigma,\mathrm{site}} + \vec j_i^{\Sigma,\mathrm{site}})
    &= r_i^\Sigma(\vec c^\Sigma) + (c_i \vec v + \vec j_i) \cdot \vec n,
    &i = 1, \ldots, N, \, t \geq 0, \, \vec z \in \Sigma,
  \end{align*}
rather the following version will be considered throughout the paper:
  \begin{align*}
   \partial_t c_i + \dv (c_i \vec v + \vec j_i)
    &= r_i(\vec c),
    &&i = 1, \ldots, N, \, t \geq 0, \, \vec z \in \Omega,
    \\
   \partial_t c_i^\Sigma + \dv_\Sigma (c_i^\Sigma \vec v^{\Sigma,\mathrm{site}} + \vec j_i^{\Sigma,\mathrm{site}})
    &= r_i^\Sigma(\vec c^\Sigma) + s_i^\Sigma(\vec c, \vec c^\Sigma),
    &&i = 1, \ldots, N, \, t \geq 0, \, \vec z \in \Sigma,
    \\
   (c_i \vec v + \vec j_i) \cdot \vec n
    &= s_i^\Sigma(\vec c, \vec c^\Sigma),
    &&i = 1, \ldots, N, \, t \geq 0, \, \vec z \in \Sigma.
  \end{align*}
 A schematic sketch of a reaction-diffusion-sorption system is provided by Figure~\ref{fig:diagram_sketch}.
  \begin{figure}
	\centering
	\includegraphics[scale = 1]{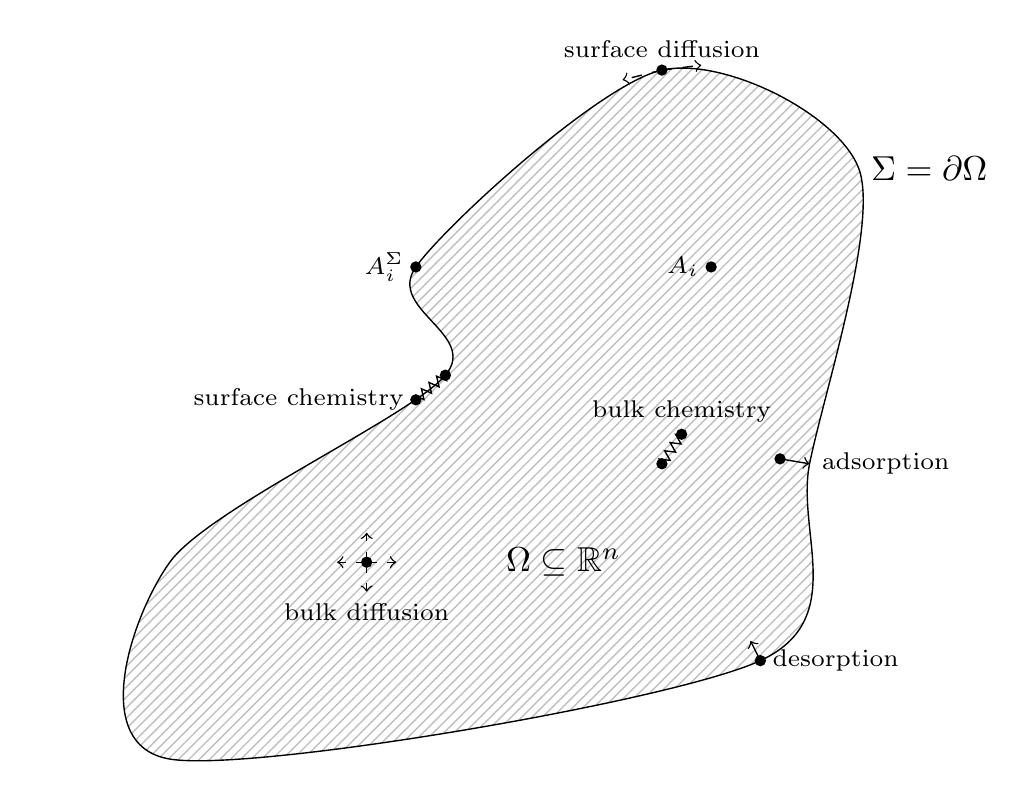}
	\caption{Physical and chemical mechanisms in a bulk-surface reaction diffusion system.}
	\label{fig:diagram_sketch}
 \end{figure}
 For a more complete thermodynamic modelling of heterogeneous catalysis, including linear momentum and internal energy balance, see \cite{SoOrMaBo19} and the references given there.
 
 \subsection{On the convective part of mass-transport}
 The total mass density of a fluid evolves w.r.t.\ the continuity equation
  \[
   \partial_t \rho + \dv (\rho \vec v)
    = 0,
    \quad
    t \geq 0, \, \vec z \in \Omega.
  \]
 Thus, for a thermodynamically fully consistent model, there is a strong interplay between the  barycentric velocity, which evolution depends on the total mass density $\rho = \sum_{i} \rho_i$ of the fluid, and the individual mass balance, which again involves the barycentric velocity, so that the barycentric velocity influences, but depends on the change of mass density due to diffusion and chemical reactions as well.
 In particular, a full model needs to include not only the solutes which are dissolved within a solvent, but the solvent itself as well.
 This makes such models highly complicated with delicate interaction between the several evolution mechanisms, and we will abstain from this most general setting.
 Instead, we consider only the evolution of the mass densities of the dissolved solvents $A_i$, $i = 1, \ldots, N$.
 Consequently, the model cannot be expected to respect, e.g., conservation of linear momentum anymore, since any interaction between the solutes and the solvent is neglected.
 Moreover, the barycentric velocity is only the barycentric velocity of the solutes instead of the whole fluid mixture.
 Therefore, suitable models for the barycentric velocity (of the solutes) are required.
 These may not be subject to momentum conservation, but only allow for some weakened form of thermodynamic consistency.
 Possible approaches are the following:
 \begin{enumerate}
  \item
   \emph{Darcy's law}: A common approximation used especially for porous media is the Darcy law
    \[
     \vec v
      = - \alpha \nabla p
    \]
   for the pressure $p$ which by the Gibbs-Duhem relation is given by
    \[
     p
      = \sum_{i=1}^N c_i \mu_i - c \psi
    \]
   with $\mu_i$ denoting the (molar) chemical potentials, and $\psi$ the (molar) free energy.
   In the isothermal case, the gradient of the pressure is directly related to the gradient of the chemical potentials via
    \[
     \nabla p
      = \sum_{i=1}^N c_i \nabla \mu_i.
    \]
   This approach, as means to eliminate the momentum equation in a reasonable (for certain cases) way, has recently been employed in \cite{DruJun19+}.
   It leads to the velocity model
    \[
     \vec v
      = - \alpha \sum_{i=1}^N c_i \nabla \mu_i.
    \]
,
  \item
   \emph{Experimental observations}: As mentioned in \cite{Haa63}, the molar averaged velocity comes closest to the (weighted) averaged velocity which can be experimentally measured. If one then restricts to systems at rest (in the sense of vanishing molar averaged velocity), i.e.\ $\vec v^\mathrm{mol} \approx \vec 0$, one may derive the barycentric velocity from this as
    \[
     \vec v
      = - \sum_{i=1}^N \frac{1}{c M_i} \vec J_i.
    \]
  \item
   \emph{Perturbation approach}: 
   If the barycentric velocity $v$ and its divergence field $\dv v$ obey certain bounds,
    \[
     \dv(c_i \vec v)
      = \vec v \cdot \nabla c_i + (\dv \vec v) c_i
    \]
   can be treated as a perturbation of the second-order differential operator corresponding to diffusive transport.
   We then consider $\vec v$ as an additional perturbation which then may be brought to the right-hand side of the balance equation for the molar concentrations $c_i$.
   Therefore, in derivation and mathematical analysis on the surface we restrict ourselves to the unperturbed fundamental case $\vec v = 0$.
 \end{enumerate}
Similarly, to have momentum conservation of the following limit models, one would need to take the particles constituting the solid into account.
Since we are looking for models of reasonable complexity, on the surface we do not consider the molar fluxes w.r.t.\ the barycentre, but rather focus on relative velocities w.r.t.\ the site averaged velocity of the surface fluid, so that in particular $\sum_{i=0}^N \vec j_i^{\Sigma,\mathrm{site}} = \vec 0$, and the further modelling and analysis will be done only for the unperturbed case of vanishing site averaged surface velocity $\vec v^{\Sigma,\mathrm{site}} = \vec 0$.

 \subsection{Reduced thermodynamic consistency}
 When considering the prototypical situation of vanishing barycentric velocity $\vec v \equiv \vec 0$ in the bulk and site averaged velocity $\vec v^{\Sigma,\mathrm{site}} = \vec 0$ on the surface, the reduced model cannot be expected to satisfy conservation of (linear) momentum, especially if one thinks of the exchange of momentum due to sorption processes at the surface.
 Since the solvent $A_{N+1}$ and the solid surface are not included in the model, this does not come as a surprise. However, one sensibly may ask which thermodynamic principles are satisfied yet, so that one may hopefully consider a subclass of models for the closure relations for diffusion, chemical reaction, adsorption and desorption, which, though not fully thermodynamically consistent, still satisfy some reduced form of thermodynamic consistency.
 \newline
 In this manuscript, we consider isothermal and isobaric systems, for which the second law of thermodynamics (entropy law) implies a \emph{free energy} dissipation principle, namely
  \begin{principle}[Free energy minimisation]
   Every thermodynamically consistent system which is isothermal and isobaric seeks to minimise its total free energy $F$, i.e.\ formally
    \[
     \frac{\dd F}{\dd t}
      \leq 0.
    \]
  \end{principle}
 Since the surface $\Sigma = \partial \Omega$ is assumed to be static, i.e.\ does not move or evolve in time, we may define the total free energy of the coupled bulk-surface system as
  \[
   F(t)
    = \int_\Omega (\rho \psi)(t,\vec z) \dd \vec z
     + \int_\Sigma (\rho^\Sigma \psi^\Sigma)(t,\vec z) \dd \sigma(\vec z)
    = \int_\Omega (c \psi^\mathrm{mol})(t,\vec z) \dd \vec z
     + \int_\Sigma (c^\Sigma \psi^{\Sigma,\mathrm{mol}})(t,\vec z) \dd \sigma(\vec z).
  \]
 Here, $\psi$ and $\psi^\Sigma$ denotes the \emph{specific free energy densities} in the bulk and on the surface, resp., and $\psi^\mathrm{mol}$, resp.\ $\psi^{\Sigma,\mathrm{mol}}$, the molar free energy density.
 The free energy density will be modelled such that $\rho \psi$ can be expressed as $c \psi^\mathrm{mol} = f(\vec c) = f(c_1, \ldots, c_N)$ and $c^\Sigma \psi^{\Sigma,\mathrm{mol}} = f^\Sigma(\vec c^\Sigma) = f^\Sigma(c_0^\Sigma, c_1^\Sigma, \ldots, c_N^\Sigma)$, i.e.\ for the surface free energy density, the vacancies are treated as an independent species, though, under the constraint $\sum_{i=0}^N c_i^\Sigma = c_S^\Sigma$.
 (For deriving a model for the surface diffusion this is a convenient choice as it implies some symmetry properties of the Fick-Onsager coefficients; see the next subsection.)
 The \emph{molar chemical potentials} $\mu_i^\mathrm{mol}$ in the bulk and $\mu_i^{\Sigma,\mathrm{mol}}$ on the surface are then given by
  \[
   \mu_i^\mathrm{mol}
    = \frac{\partial f}{\partial c_i}(\vec c),
    \quad i = 1, \ldots, N
    \quad \text{and} \quad
   \mu_i^{\Sigma,\mathrm{mol}}
    = \frac{\partial f^\Sigma}{\partial c_i^\Sigma}(\vec c^\Sigma),
    \quad i = 0, 1, \ldots, N.
  \]
 Inserting the evolution equations for $\vec c$ and $\vec c^\Sigma$, and using the model approximation $\vec v = \vec 0$ and $\vec v^{\Sigma,\mathrm{site}} = \vec 0$, one may compute (using the divergence theorem in the bulk and on the surface as well as $s_0^\Sigma = - \sum_{i=1} s_i^\Sigma$) the change of total free energy formally as
  \begin{align*}
   \frac{\dd F}{\dd t}
    &= \frac{\dd}{\dd t} \big( \int_\Omega (c \psi^\mathrm{mol})(t,\vec z) \dd z
     + \int_\Sigma (c^\Sigma \psi^{\Sigma,\mathrm{mol}} \dd \sigma(\vec z) \big)
     \\
    &= \int_\Omega \sum_{i=1}^N (\mu_i r_i + \vec j_i^{\Sigma,\mathrm{site}} \cdot \nabla_\Sigma \mu_i^\Sigma)(t,\vec z) \dd \vec z
     + \int_\Sigma \sum_{i=0}^N (\mu_i^\Sigma r_i^\Sigma + \vec j_i^{\Sigma,\mathrm{site}} \cdot \nabla_\Sigma \mu_i^\Sigma)(t, \vec z) \dd \sigma(\vec z)
     \\ &\quad
     + \int_\Sigma \sum_{i=1}^N (\mu_i^\Sigma - \mu_0^\Sigma - \mu_i) s_i^\Sigma.
  \end{align*}
 Similar to the formulation of the entropy production rate for fully thermodynamically consistent models, cf.\ \cite{BoDr15}, the dissipative terms may be regrouped as follows:
  \begin{align*}
   \frac{\dd F}{\dd t}
    &= \int_\Omega \sum_{i=1}^N \mu_i r_i \dd \vec z
     + \int_\Omega \sum_{i=1}^N \vec j_i \cdot \nabla \mu_i \dd \vec z
      \\ &\quad
     + \int_\Sigma \sum_{i=0}^N \mu_i^\Sigma r_i^\Sigma \dd \sigma(\vec z)
     + \int_\Sigma \sum_{i=0}^N \vec j_i^{\Sigma,\mathrm{site}} \cdot \nabla_\Sigma \mu_i^\Sigma \dd \sigma(\vec z)
     + \int_\Sigma \sum_{i=1}^N (\mu_i^\Sigma - \mu_0^\Sigma - \mu_i) s_i^\Sigma
     \\
    &= - \int_\Omega \big( \zeta^\mathrm{chem} + \zeta^\mathrm{diff} \big) \dd \vec z
     - \int_\Sigma \big( \zeta^{\Sigma,\mathrm{chem}} + \zeta^{\Sigma,\mathrm{diff}} + \zeta^{\Sigma,\mathrm{sorp}} \big) \dd \sigma(\vec z),
  \end{align*}
 where
  \begin{align*}
   \zeta^\mathrm{chem}
    &= - \sum_{i=1}^N \mu_i r_i,
    \\
   \zeta^\mathrm{diff}
    &= - \sum_{i=1}^N \vec j_i \cdot \nabla \mu_i,
    \\
   \zeta^{\Sigma,\mathrm{chem}}
    &= - \sum_{i=0}^N \mu_i^\Sigma r_i^\Sigma,
    \\
   \zeta^{\Sigma,\mathrm{diff}}
    &= \sum_{i=0}^N \mu_i^\Sigma r_i^\Sigma
    \quad \text{and}
    \\
   \zeta^{\Sigma,\mathrm{sorp}}
    &= - \sum_{i=1}^N (\mu_i^\Sigma - \mu_0^\Sigma - \mu_i) s_i^\Sigma.
  \end{align*}
 We restrict ourselves to particular subclasses of closure relations, namely such that $\vec r = (r_i)_{i=1,\ldots,N} = \vec r(\vec \mu)$, $\bb J = (\vec j_i)_{i=1,\ldots,N} = \bb J(\nabla \vec \mu)$, $\vec r^\Sigma = (r_i^\Sigma)_{i=0,\ldots,N} = \vec r^\Sigma(\vec \mu^\Sigma)$ and $\vec s^\Sigma = (s_i^\Sigma)_{i=1,\ldots,N} = \vec s^\Sigma(\vec \mu_\mathrm{eff}^\Sigma - \vec \mu)$, $\bb J^\Sigma = (\vec j_i^\Sigma)_{0=1,\ldots,N} = \bb J(\nabla_\Sigma \vec \mu^\Sigma)$, i.e.\
  \begin{enumerate}
   \item
    the chemical reaction rates $r_i$ in the bulk are determined by the (vector of) chemical potentials $\vec \mu = (\mu_1, \ldots, \mu_N)^\mathsf{T}$;
   \item
    analogously, the chemical reaction rates $r_i^\Sigma$ on the surface are a function of the surface chemical potentials $\vec \mu^\Sigma = (\mu_0^\Sigma, \mu_1^\Sigma, \ldots, \mu_N^\Sigma)^\mathsf{T}$;
   \item
    diffusive fluxes $\vec j_i$ in the bulk can be expressed in terms of the gradient of chemical potentials $\nabla \vec \mu = (\nabla \mu_1, \ldots, \nabla \mu_N)^\mathsf{T}$;
   \item
    similarly, surface diffusive fluxes $\vec j_i^{\Sigma,\mathrm{site}}$ depend on the surface gradient of the surface chemical potentials $\nabla_\Sigma \vec \mu^\Sigma = (\nabla_\Sigma \mu_0^\Sigma, \ldots, \nabla_\Sigma \mu_N^\Sigma)^\mathsf{T}$;
   \item
    finally, the sorption rates $s_i^\Sigma$ are models as functions of the difference of chemical potentials $\vec \mu_\mathrm{eff}^\Sigma - \vec \mu = (\mu_i^\Sigma - \mu_0^\Sigma - \mu_i)_{i=1,\ldots,N}$ with the effective surface chemical potentials $\mu_{\mathrm{eff},i}^\Sigma = \mu_i^\Sigma - \mu_0^\Sigma$.
  \end{enumerate}
 To ensure that there is dissipation of free energy for this particular class, we model the closure relations such that the contribution to the change of free energy from every thermodynamic subprocess is non-negative, i.e.\ the terms $\zeta^\mathrm{diff}$, $\zeta^\mathrm{chem}$, $\zeta^{\Sigma,\mathrm{diff}}$, $\zeta^{\Sigma,\mathrm{chem}}$ and $\zeta^{\Sigma,\mathrm{sorp}}$ all become non-negative for every thermodynamically feasible process.
  \begin{remark}
   The factor $-1$ and the symbol $\zeta$ for the several contributions to the consumption of free energy hint at the corresponding contribution to the entropy production for the respective thermodynamic subprocesses.
   For more details, we refer to the manuscript \cite{AuBo19+1}.
  \end{remark}
 In the next subsection, we restrict ourselves to specific models for these closure relations. 
 
 \subsection{Models for bulk and surface diffusion, chemical reactions and sorption processes}
 Within this subsection, the underlying models for the diffusion processes in the bulk and on the surface, as well as for the chemical reactions and the sorption at the boundary are introduced. The time scale analysis will be based on these particular models, however it should be clear how to extend the arguments to the case of other types of models for the diffusion and/or chemical reactions, in particular to other models for the chemical potentials in the bulk or on the surface.
 The motivation and derivation of certain limit models will be established via the special case of Fickian diffusion in the bulk and a (single-site) multi-component Langmuir model on the surface and Fick--Onsager surface diffusion.
 More precisely, in the bulk $N$ chemical substances $A_i$, $i = 1, \ldots, N$, are considered which are dissolved in a solvent $A_{N+1}$. The concentrations $c_i$ of the solutes $A_i$ are assumed to be much smaller than the concentration of the solvent $A_{N+1}$, i.e.\ $c_i \ll c_{N+1}$, and, moreover, it is assumed that the solutes merely interact with the solvent, and interactions between distinct solutes can be neglected.
 Standard Fickian diffusion $\vec j_i = - d_i \nabla c_i$, where typically $d_i > 0$ depends on the distribution of $\vec c$, can neither constitute a thermodynamic consistent model for the diffusive fluxes nor is it consistent with the constraint $\sum_{i=1}^N \vec J_i = \vec 0$.
 Keep in mind that the convective flux $c_i \vec v$ is treated as a perturbation and the basic analysis here will be restricted to the fundamental case $\vec v = \vec 0$. 
 For dilute systems, however, it is still a reasonably good approximation to more general and thermodynamically consistent diffusion models such as Fick-Onsager diffusion or Maxwell--Stefan diffusion models; cf.\ the representation of the inversion of the Maxwell--Stefan relations in \cite[Lemma 2.2]{HeMePrWi17}.
 In this simplified situation, $d_i > 0$ will be further assumed to be constant, so that the bulk reaction-diffusion system takes the form
  \[
   \partial_t c_i - d_i \Delta c_i
    = r_i(\vec c),
    \quad
    i = 1, \ldots, N, \, t \geq 0, \, \vec z \in \Omega.
  \]
On the surface, however, low surface concentrations would be a highly unrealistic assumption, so that Fickian diffusion is ruled out, but a Fick-Onsager diffusion model (or, a Maxwell-Stefan type model) is employed instead.
Hence, it is assumed that
 \[
  \vec j_i^{\Sigma,\mathrm{site}}
   = - \sum_{i=0}^N d_{ij}^\Sigma(\vec c^\Sigma) \nabla_\Sigma c_j^\Sigma,
   \quad
   i = 0, 1, \ldots, N,
 \]
where the Fick-Onsager surface diffusion coefficients $d_{ij}^\Sigma = d_{ji}^\Sigma$ are symmetric in $i,j \in \{0, 1, \ldots, N\}$.
 Moreover, the $d_{ij}^\Sigma$ sum up to zero, i.e.\
 \[
  \sum_{j=0}^N d_{ij}^\Sigma(\vec c^\Sigma)
   = 0,
   \quad i = 1, \ldots, N
 \]
since $\sum_{i=0}^N \vec j_i^{\Sigma,\mathrm{site}} = 0$ to have consistency with the definition of the diffusive fluxes, i.e.\ $\vec j_i^{\Sigma,\mathrm{site}} = c_i^\Sigma (\vec v_i^\Sigma - \vec v^{\Sigma,\mathrm{site}})$ for the site averaged velocity $\vec v^{\Sigma,\mathrm{site}}$.
Moreover, analogously to our treatment of the bulk phase, vanishing molar averaged velocity $\vec v^{\Sigma,\mathrm{site}} = \vec 0$ will be assumed throughout.
Hence, the surface reaction-diffusion-sorption model takes the form
 \[
  \partial_t c_i^\Sigma - \dv_\Sigma \left( \sum_{j=0}^N d_{ij}^\Sigma(\vec c^\Sigma) \nabla_\Sigma c_j^\Sigma \right)
   = r_i(\vec c^\Sigma) + s_i(\vec c|_\Sigma, \vec c^\Sigma),
   \quad
   i = 0, 1, \ldots, N.
 \]
It remains to model the rate of chemical reactions in the bulk ($r_i$) and on the surface ($r_i^\Sigma$) as well as the sorption rates ($s_i^\Sigma$).
 \begin{remark}
  Throughout, we set
   \[
    r_0^\Sigma(\vec c^\Sigma)
     = - \sum_{i=1}^N r_i(\vec c^\Sigma),
     \quad
    s_0^\Sigma(\vec c|_\Sigma, \vec c^\Sigma)
     = - \sum_{i=1}^N s_i^\Sigma(\vec c|_\Sigma, \vec c^\Sigma).
   \]
  Then the balance equation for the vacancy concentrations $c_0^\Sigma$ follows from the balance equations for the adsorbates $A_i^\Sigma$, $i = 1, \ldots, N$.
 \end{remark}
To find a suitable mathematical model for the bulk chemistry, one typically starts with a set of formal reaction equations
 \[
  \sum_{i=1}^N \alpha_i^a A_i
   \rightleftharpoons \sum_{i=1}^N \beta_i^a A_i,
   \quad
   a = 1, \ldots, m
 \]
and uses the (molar) chemical potentials $\mu_i = \mu_i^0(T) + RT \ln \frac{c_i}{c_i^\ast}$, $i = 1, \ldots, N$, which fit well to the Fickian modelling of the bulk diffusion.
Note that, here, $\mu_i^0(T)$ is temperature-dependent but independent of the vector of concentrations $\vec c$.
Moreover, $R$ is the universal gas constant, $T > 0$ the absolute temperature which is assumed to be constant here (\emph{isothermal case}) and $\vec c^\ast = (c_1^\ast, \ldots, c_N^\ast)^\mathsf{T} \in (0,\infty)^N$ is some reference concentration (depending on the choice of $\mu_i^0(T)$).
 \begin{remark}[Free energy and pressure in the bulk]
  The (molar) chemical potentials $\mu_i(T) = \mu_i^0(T) + \ln x_i$, where $x_i = c_i / c_i^\ast$ correspond to free energy density in the bulk of the form
   \[
    c \psi^\mathrm{mol}
     = - p_0 + \sum_{i=1}^N c_i \left[ \mu_i^0(T) + RT (\ln x_i - 1) \right].
   \]
  By the Gibbs-Duham relation, the corresponding pressure is
   \[
    p
     = \sum_{i=1}^N c_i \mu_i - c \psi^\mathrm{mol}
     = p_0 + RT \sum_{i=1}^N c_i.
   \]
 \end{remark}
By the second law of thermodynamics (entropy principle) or the above reduced principle of free-energy minimisation, for the thermodynamic subprocess of bulk chemistry one should have $\zeta^\mathrm{chem} = \mathcal{A}_a R_a(\vec c) \geq 0$ for the affinity
 \[
  \mathcal{A}_a
   = \sum_{i=1}^N \mu_i \nu_i^a
   = \sum_{i=1}^N (\mu_i^0 + RT \ln \frac{c_i}{c_i^\ast}) (\beta_i^a - \alpha_i^a)
  \]
of the $a$-th reaction. The reaction velocity of the $a$-th reaction is modelled as the difference of forward and backward reaction velocities $R_a = R_a^f - R_a^b$, so that we assume
 \[
  (R_a^f - R_a^b) \sum_{i=1}^N (\mu_i^0 + RT \ln \frac{c_i}{c_i^\ast}) (\beta_i^a - \alpha_i^a)
  \geq 0.
 \]
In this paper, \emph{mass-action kinetics} is used as a model for chemical reactions in the bulk.
The latter may be derived as follows:
For one, say the forward, direction of the reversible chemical reactions, a law of the type
 \[
  R_a^f
   = k_a^f \prod_{i=1}^N c_i^{\alpha_i^a}
   =: k_a^f \vec c^{\vec \alpha^a},
   \quad
   a = 1, \ldots, m
 \]
with some $k_a^f > 0$, typically depending on $T$ and $\vec c$, is assumed (ansatz of reactive collisions).
Then, to ensure positivity of the binary product $\mathcal{A}_a R_a$, $R_a^b$ is determined by a logarithmic closure relation, viz.\
 \[
  \ln \frac{R_a^f}{R_a^b}
   = - \gamma^a \frac{1}{RT} \mathcal{A}_a
   \quad \text{with } \gamma^a > 0.
 \]
Note that, typically, the mixture is far away from chemical equilibrium, so that a linear closure is inappropriate and a logarithmic closure is used instead.
In what follows, we let $\gamma^a = 1$ which is sufficient for our modelling purpose.
This leads to the effective reaction velocity $R_a = k_a^f \vec c^{\vec \alpha^a} - k_a^b \vec c^{\vec \beta^a}$ with
 \[
  \frac{k_a^b}{k_a^f}
   = (\vec \mu^0)^{\vec \nu^a}
   \quad \text{for} \quad
  \vec \mu^0
   = (\mu_1^0, \ldots, \mu_N^0)^\mathsf{T},
 \]
which is exactly the standard mass-action kinetics law.
Note that, w.l.o.g., one may assume $c_i^\ast = 1$ (with physical dimension, though) after replacing $\mu_i^0$ by $\tilde \mu_i^0 = \mu_i^0 - RT \ln c_i^\ast$.
This convention will tacitly be used from now on.
\newline
On the surface, a realistic model for the chemical potentials should include the available space on the surface as well.
One possibility to construct a model for the chemical potential is to assume that there is a maximal capacity $c_S^\Sigma > 0$ on the surface, as only limited space it available for the adsorbed species. 
This imposes the constraint
 $
  \sum_{i=1}^N c_i \leq c_S^\Sigma
 $
 on the surface concentrations.
 Interpreting the \emph{free sites} (or \emph{vacancies}) as an additional species $A_0^\Sigma$ (see Figure~\ref{fig:diagram_free-sites})
 \begin{figure}
	\centering
	\includegraphics[scale = 1]{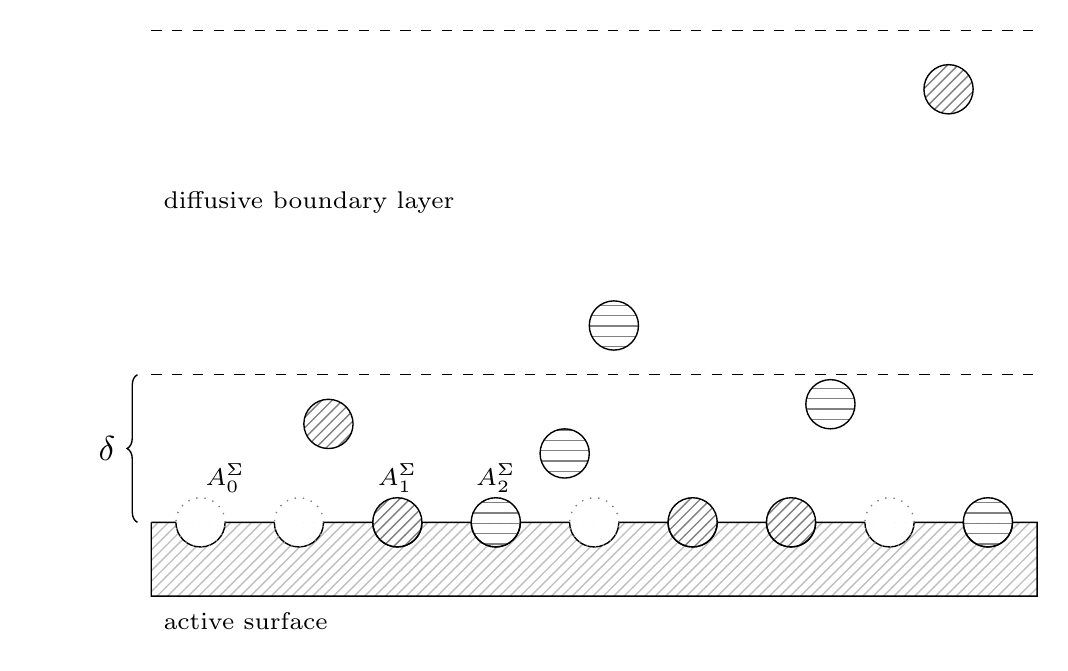}
	\caption{Free sites on the surface are interpreted as an additional species $A_0^\Sigma$.}
	\label{fig:diagram_free-sites}
 \end{figure}
and denoting by $c_0^\Sigma := c_S^\Sigma - \sum_{i=1}^N c_i^\Sigma$ the \emph{vacancy concentrations}, this constraint can be reformulated as
 \[
  \sum_{i=0}^N c_i = c_S^\Sigma
   \quad \text{and} \quad
   c_i \geq 0,
   \quad
   i = 0, 1, \ldots, N.
 \]
 Introducing the \emph{occupancy numbers} $\theta_i := c_i^\Sigma / c_S^\Sigma$, $i = 0,1, \ldots, N$, this constraint may be equivalently expressed as $\sum_{i=0}^N \theta_i = 1$.
 The surface free energy $c^\Sigma \psi^{\Sigma,\mathrm{mol}}$ is modelled as that of an ideal mixture of the species $A_0^\Sigma$, \ldots, $A_N^\Sigma$, i.e.\ (Langmuir model)
  \[
   c^\Sigma \psi^{\Sigma,\mathrm{mol}}
    = f^\Sigma(\vec c^\Sigma)
    = - p_0^\Sigma + \sum_{i=0}^N c_i^\Sigma (\mu_i^{\Sigma,0}(T) + \ln \theta_i)
  \]
 for some temperature-dependent parameters $\mu_i^{\Sigma,0}(T)$.
 One should keep in mind that $c_0^\Sigma = c_S^\Sigma - \sum_{i=1}^N c_i^\Sigma$ is not an independent variable, but nevertheless we assign a chemical potential $\mu_0^\Sigma$ to the vacancies $A_0^\Sigma$, so that
  \[
   c^\Sigma \psi^{\Sigma,\mathrm{mol}}
    = f^\Sigma (T, c_0^\Sigma, c_1^\Sigma, \ldots, c_N^\Sigma).
  \]
 The surface chemical potentials $\mu_i^\Sigma = \frac{\partial (c^\Sigma \psi^{\Sigma,\mathrm{mol}})}{\partial c_i^\Sigma}$, $i = 0, 1, \ldots, N$, are then modelled as
 \[
  \mu_i^\Sigma
   = \mu_i^{\Sigma,0}(T) + RT \ln \theta_i,
   \quad
   i = 0, 1, \ldots, N.
 \]
 Moreover, by the Gibbs-Duhem relation, one may compute the surface pressure as
  \[
   p^\Sigma
    = p_0^\Sigma + RT c_S^\Sigma.
  \]
 Considerations similar to those for the bulk chemistry, transferred to the surface chemical reactions
 \[
  \sum_{i=1}^N \alpha_i^{\Sigma,a} A_i^\Sigma
   \rightleftharpoons \sum_{i=1}^N \beta_i^{\Sigma,a} A_i^\Sigma,
   \quad
   a = 1, \ldots, m^\Sigma,
 \]
then lead to chemical reaction rates of the form
 \[
  R_a^\Sigma
   = k_a^{\Sigma,f} \vec \theta^{\vec \alpha^{\Sigma,a}} - k_a^{\Sigma,b} \vec \theta^{\vec \beta^{\Sigma,a}},
   \quad
   a = 1, \ldots, m^\Sigma.
 \]
To do so, we interpret surface chemical reactions $\sum_{i=1}^N \alpha_i^{\Sigma,a} A_i^\Sigma \rightleftharpoons \sum_{i=1}^N \beta_i^{\Sigma,a} A_i^\Sigma$ rather as chemical reactions of the type $\sum_{i=0}^N \alpha_i^{\Sigma,a} \rightleftharpoons \sum_{i=0}^N \beta_i^{\Sigma,a} A_i^\Sigma$, taking conservation of total sites into account and defining
 \[
  \alpha_0^{\Sigma,a}
   = \begin{cases}
    \sum_{i=0}^N \nu_i^{\Sigma,a},
    &\sum_{i=0}^N \nu_i^{\Sigma,a} \geq 0,
    \\
    0,
    &\text{else},
   \end{cases}
   \quad
  \beta_0^{\Sigma,a}
   = \begin{cases}
    - \sum_{i=0}^N \nu_i^{\Sigma,a},
    &\sum_{i=0}^N \nu_i^{\Sigma,a} \leq 0,
    \\
    0,
    &\text{else}
   \end{cases}
 \]
so that $\sum_{i=0}^N \nu_i^{\Sigma,a} = 0$ for $a = 1, \ldots, m^\Sigma$.
Here, $\vec \theta = (\theta_0, \vec \theta_\mathrm{red}) = (\theta_0, \ldots, \theta_N)^\mathsf{T}$, $\vec \alpha^{\Sigma,a} = (\alpha_0^{\Sigma,a}, \ldots, \alpha_N^{\Sigma,a})^\mathsf{T}$ and $\vec \beta^{\Sigma,a} = (\beta_0^{\Sigma,a}, \ldots, \beta_N^{\Sigma,a})^\mathsf{T}$.
 Moreover, the logarithmic closure gives the relations $k_a^{\Sigma,b} / k_a^{\Sigma, f} = (\vec \mu^{\Sigma,0})^{\vec \nu^{\Sigma,a}}$.
 Finally, the sorption rates $s_i^\Sigma(\vec c, \vec c^\Sigma)$ are modelled as reaction rates of chemical reactions of the form
  \[
   A_i 
    \rightleftharpoons A_i^\Sigma,
    \quad \text{or, taking site conservation into account,} \quad
   A_i + A_0^\Sigma
    \rightleftharpoons A_i^\Sigma.
  \]
 This leads to an effective reaction velocity of the form
  \[
   R_i^\mathrm{sorp}
    = R_i^\mathrm{ad} - R_i^\mathrm{de}
    = k_i^\mathrm{ad} c_i \theta_0 - k_i^\mathrm{de} c_i^\Sigma,
    \quad
    i = 1, \ldots, N,
  \] 
 where $\ln \left( \frac{R_i^\mathrm{ad}}{R_i^\mathrm{de}} \right) = \mu_i^\Sigma - \mu_0^\Sigma - \mu_i|_\Sigma$.
 To summarise, the bulk and surface chemistry, and the sorption rates are modelled as follows:
  \begin{align*}
   r_i(\vec c)
    &= \sum_{a=1}^m \nu_i^a (k_a^f \vec c^{\vec \alpha^a} - k_a^b \vec c^{\vec \beta^a}),
    &&i = 0, 1, \ldots, N,
    \\
   r_i^\Sigma(\vec c^\Sigma)
    &= \sum_{a=1}^{m^\Sigma} \nu_i^{\Sigma,a} (k_a^{\Sigma,f} \vec \theta^{\vec \alpha_\mathrm{ext}^{\Sigma,a}} - k_a^{\Sigma,b} \vec \theta^{\vec \beta_\mathrm{ext}^{\Sigma,a}}) c_S^\Sigma,
    &&i = 0, 1, \ldots, N
    \\
   s_i^\Sigma(\vec c^\Sigma)
    &= (k_i^\mathrm{ad} c_i \theta_0 - k_i^\mathrm{de} \theta_i) c_S^\Sigma,
    &&i = 1, \ldots, N.
  \end{align*}
 In the case of vanishing convective fluxes, the condensed form of the full heterogeneous reaction-diffusion-sorption model reads as
  \begin{align*}
   \partial_t c_i - d_i \Delta c_i
    &= \sum_{a} \nu_i^a (k_a^f \vec c^{\vec \alpha^a} - k_a^b \vec c^{\vec \beta^a})
    &&\text{in } (0, \infty) \times \Omega,
    \\
   \partial_t c_i^\Sigma + \dv_\Sigma \vec j_i^{\Sigma,\mathrm{mol}}
    &= \sum_{a} \nu_i^{\Sigma,a} c_S^\Sigma (k_a^{\Sigma, a} \vec \theta^{\vec \alpha_\mathrm{ext}^{\Sigma,a}} - k_a^{\Sigma, a} \vec \theta^{\vec \beta_\mathrm{ext}^{\Sigma,a}})
     + c_S^\Sigma (k_i^{\mathrm{ad}} c_i \theta_0 - k_i^{\mathrm{de}} c_i^\Sigma)
    &&\text{on } (0, \infty) \times \Sigma,
    \\
   - d_i \nabla c_i \cdot \vec n
    &= c_S^\Sigma (k_i^{\mathrm{ad}} c_i|_\Sigma \theta_0 - k_i^{\mathrm{de}} \theta_i)
    &&\text{on } (0, \infty) \times \Sigma
  \end{align*}
 with $\vec j_i^{\Sigma,\mathrm{site}} = - \sum_{j=0}^N d_{ij}^\Sigma(\vec c^\Sigma) \nabla_\Sigma c_j$, for $i = 1, \ldots, N$.
 
 \begin{remark}
  One may as well consider $c_0^\Sigma = c_S^\Sigma - \sum_{i=1}^N c_i^\Sigma$ and $\theta_0 = 1 - \sum_{i=1}^N \theta_i$ as dependent functions of the $c_i^\Sigma$'s resp.\ $\theta_i$', so that one might argue that the correct choice for the chemical potentials is
   \[
    \tilde \mu_i^\Sigma
     = \frac{\partial}{\partial c_i^\Sigma} \big( f(c_S^\Sigma - \sum_{j=1}^N c_j^\Sigma, c_1^\Sigma, \ldots, c_N^\Sigma) \big)
     = \mu_i^\Sigma - \mu_0^\Sigma.
   \]
  For this choice of the chemical potentials, however, the same reaction rate models result, as then $A_0^\Sigma$ is not treated as an independent species.
  The disadvantage of this approach is the fact that the diffusive surface fluxes in the Fick--Onsager model satisfy certain symmetry conditions due to the choice that one models the diffusive surface fluxes relative to the site averaged molar flux.
  On the other hand, the equilibrium condition at the surface would have the simpler form $\mu_i|_\Sigma = \tilde \mu_i^\Sigma$ instead of the form $\mu_i|_\Sigma = \mu_i^\Sigma - \mu_0^\Sigma$ for $i = 1, \ldots, N$.
 \end{remark}

 \subsection{Dimensionless formulation of the bulk-surface reaction-diffusion-sorption model}
 Typically, the different thermodynamic processes, i.e.\ bulk and surface diffusion, bulk and surface chemical reactions as well as sorption processes, occur on different time scales. Indeed, quite often, especially for heterogeneous catalysis models, sorption processes and chemical reactions on the surface are considerably faster then all the other thermodynamic processes. From this perspective, it is natural to consider the fast sorption or fast surface chemistry limit models of the general reaction diffusion sorption model, so that the surface diffusion may (partly) be replaced by quasi-static relations between $\vec c|_\Sigma$ and $\vec c^\Sigma$, as will be seen below.
 
 To derive a reduced limit model, one starts by establishing a \emph{dimensionless formulation} of the general model, here the reaction-diffusion-sorption model with Fickian diffusion in the bulk and a Langmuir model on the surface. For this purpose, one introduces the following characteristic quantities:
  \begin{itemize}
   \item
    $\tau^\mathrm{R} > 0$ with unit $[\tau^\mathrm{R}] = \mathrm{s}$, a \emph{characteristic time} or \emph{accumulation time}, referring to a typical time on which significant changes in the concentration profiles in the bulk and on the surface occur;
   \item
    $L_R, L_R^\Sigma > 0$ with $[L_R] = [L_R^\Sigma] = \mathrm{m}$, \emph{characteristic lengths} in the bulk and on the surface, referring to typical lengths over which differences in the concentration profile in the bulk or on the surface can be observed;
   \item
    $D_R, D_R^\Sigma > 0$ with $[D_R] = [D_R^\Sigma] = \mathrm{m}^2 / \mathrm{s}$, \emph{characteristic diffusivities} in the bulk and on the surface, referring to typical values of the diffusion coefficients $d_i$ in the bulk and $d_{ij}^\Sigma(\vec \theta)$ on the surface;
   \item
    $c_R > 0$ with $[c_R] = \mathrm{mol} \cdot \mathrm{m}^{-3}$, a \emph{characteristic concentration} in the bulk, referring to typical values of the molar concentrations $c_i$ in the bulk;
   \item
    on the surface, the maximal capacity $c_S^\Sigma$ (here assumed to be constant) with $[c_S^\Sigma] = \mathrm{mol} \cdot \mathrm{m}^{-2}$ naturally serves as \emph{area characteristic concentration};
   \item
    $k_{R,i}^\mathrm{ad}, k_{R,i}^\mathrm{de} > 0$ with $[k_{R,i}^\mathrm{ad}] = \mathrm{m}^3 \cdot \mathrm{s}^{-1} \cdot \mathrm{mol}^{-1}$ and $[k_R^\mathrm{de}] = \mathrm{s}^{-1}$, \emph{characteristic adsorption} and \emph{desorption parameters}, referring to typical values of the reaction coefficients $k_i^\mathrm{ad}$ and $k_i^\mathrm{de}$, respectively;
   \item
    $k_{R,a}^f, k_{R,a}^b, k_{R,a}^{\Sigma,f}, k_{R,a}^{\Sigma, b}$, \emph{characteristic reaction constants} for forward and backward chemical reactions in the bulk and on the surface, respectively.
  \end{itemize}

 With these characteristic parameters at hand, one may define the following dimensionless variables and parameters
  \begin{itemize}
   \item
    $t^\ast = t / \tau^\mathrm{R}$;
    \vspace{0.1cm}
   \item
    $\vec z^\ast = \vec z / L_R$, \, $\vec z^{\Sigma, \ast} = \vec z / L_R^\Sigma$;
    \vspace{0.1cm}
   \item
    $c_i^\ast = c_i / c_R$ and $\vec c^\ast = (c_i^\ast)_{i=1}^N$;
    \vspace{0.1cm}
   \item
    $\theta_i = c_i^\Sigma / c_S^\Sigma$ and $\vec \theta = (\theta_i)_{i=0}^N$;
    \vspace{0.1cm}
   \item
    $d_i^\ast = d_i / D_R$, and $d_{ij}^{\Sigma,\ast} = d_{ij}^\Sigma / D_R^\Sigma$;
    \vspace{0.1cm}
   \item
    $k_i^{\mathrm{ad},\ast} = k_i^\mathrm{ad} / k_{R,i}^\mathrm{ad}$ and $k_i^{\mathrm{de},\ast} = k_i^\mathrm{de} / k_{R,i}^\mathrm{de}$;
    \vspace{0.1cm}
   \item
    $k_a^{f,\ast} = k_a^f / k_{R,a}^f$, $k_a^{b,\ast} = k_a^b / k_{R,a}^b$ and $k_a^{\Sigma, f, \ast} = k_a^{\Sigma, f} / k_{R,a}^{\Sigma}$, $k_a^{\Sigma, b, \ast} = k_a^{\Sigma, b} / k_{R,a}^{\Sigma}$.
    \vspace{0.1cm}
  \end{itemize}
 With a slight abuse of notation, one may also write $c_i^\ast(t^\ast, \vec z^\ast) := c_i^\ast(\tau^\mathrm{R} t^\ast, L_R \vec z^\ast)$ etc., and $\partial_{t^\ast} = \tau^\mathrm{R} \partial_t$, $\nabla^\ast = \frac{1}{L_R} \nabla$, $\Delta^\ast = \frac{1}{L_R^2} \Delta$, $\nabla_\Sigma^\ast = \frac{1}{L_R^\Sigma} \nabla_\Sigma$, and $\Delta_\Sigma^\ast = \frac{1}{(L_R^\Sigma)^2} \Delta_\Sigma$, to get the following dimensionless version of the reaction-diffusion-system with Fickian diffusion in the bulk and Fick-Onsager diffusion on the surface:
  \begin{align*}
   \frac{c_R}{\tau^\mathrm{R}} \partial_{t^\ast} c_i^\ast - \frac{D_R c_R}{L_R^2} d_i^\ast \Delta^\ast c_i^\ast
    &= \sum_{a} \nu_i^a \left( k_{R,a}^f c_R^{\abs{\vec \alpha^a}} k_a^{f,\ast} (\vec c^\ast)^{\vec \alpha^a} - k_{R,a}^b c_R^{\abs{\vec \beta^a}} k_a^{b,\ast} (\vec c^\ast)^{\vec \beta^a} \right)
    &&\text{in } (0, \infty) \times \Omega,
    \\
   \frac{1}{\tau^\mathrm{R}} \partial_{t^\ast} \theta_i - \frac{D_R^\Sigma}{(L_R^\Sigma)^2} \dv_\Sigma^\ast ( \sum_{j=0}^N d_{ij}^{\Sigma,\ast} \nabla_\Sigma^\ast \theta_j )
    &= \sum_{a} \nu_i^{\Sigma,a} (k_{R,a}^{\Sigma,f} k_a^{\Sigma,f, \ast} \vec \theta^{\vec \alpha^{\Sigma,a}} - k_{R,a}^{\Sigma,b} k_a^{\Sigma, b, \ast} \vec \theta^{\vec \beta^{\Sigma,a}})
     \\ &\qquad
     + (k_{R,i}^\mathrm{ad} c_R k_i^{\mathrm{ad}, \ast} c_i^\ast \theta_0 - k_{R,i}^\mathrm{de} k_i^{\mathrm{de}, \ast} \theta_i)
     &&\text{on } (0,\infty) \times \Sigma,
     \\
    - \frac{D_R c_R}{L_R c_S^\Sigma} \partial_{\vec n^\ast} c_i^\ast
     &= (k_{R,i}^\mathrm{ad} c_R k_i^{\mathrm{ad}, \ast} c_i^\ast c_0^\Sigma - k_{R,i}^\mathrm{ad} k_i^{\mathrm{de}, \ast} \theta_i)
     &&\text{on } (0, \infty) \times \Sigma.
  \end{align*}
 Here, the standard notation $\abs{\vec \gamma} = \sum_{i=1}^N \gamma_i$ for vectors $\vec \gamma = (\gamma_1, \ldots, \gamma_N)^\mathsf{T} \in \N_0^N$ is employed.
 We now introduce \emph{characteristic times} for the several sub-processes as follows:
  \begin{itemize}
   \item
    $\tau^\mathrm{R}$ (characteristic accumulation time),
    \vspace{0.1cm}
   \item
    $\tau_a^{\mathrm{react},f} = \frac{1}{k_{R,a}^f c_R^{\abs{\vec \alpha^a}-1}}$ and $\tau_a^{\mathrm{react},b} = \frac{1}{k_{R,a}^b c_R^{\abs{\vec \beta^a}-1}}$ (characteristic times for forward resp.\ backward chemical reactions in the bulk),
    \vspace{0.1cm}
   \item
    $\tau_a^{\mathrm{react},\vec \theta, f} = \frac{1}{k_{R,a}^{\Sigma,f}}$ and $\tau_a^{\mathrm{react},\vec \theta, b} = \frac{1}{k_{R,a}^{\Sigma,b}}$ (characteristic times for forward resp.\ backward chemical reactions on the surface),
    \vspace{0.1cm}
   \item
    $\tau^\mathrm{diff} = \frac{L_R^2}{D_R}$ (characteristic time for the bulk diffusion)
    \vspace{0.1cm}
   \item
    $\tau^{\Sigma, \mathrm{diff}} = \frac{(L_R^\Sigma)^2}{D_R^\Sigma}$ (characteristic time for the surface diffusion)
    \vspace{0.1cm}
   \item
    $\tau_i^\mathrm{ad} = \frac{1}{k_{R,i}^\mathrm{ad} c_R}$ and  $\tau_i^\mathrm{de} = \frac{1}{k_{R,i}^\mathrm{de}}$  (characteristic times for adsorption and desorption rates)
    \vspace{0.1cm}
   \item
    $\tau^\mathrm{trans} = \frac{L_R c_S^\Sigma}{D_R c_R} = \frac{c_S^\Sigma}{L_R c_R} \tau^\mathrm{diff}$ (characteristic time for transmission between bulk and surface)
    \vspace{0.1cm}
  \end{itemize}
 This leads to the formulation
  \begin{align*}
   \frac{1}{\tau_R} \partial_{t^\ast} c_i^\ast - \frac{1}{\tau^\mathrm{diff}} d_i^\ast \Delta^\ast c_i^\ast
    &= \sum_{a} \nu_i^a \left( \frac{1}{\tau_a^{\mathrm{react},f}} k_a^{f,\ast} (\vec c^\ast)^{\vec \alpha^a} - \frac{1}{\tau_a^{\mathrm{react},b}} k_a^{b,\ast} (\vec c^\ast)^{\vec \beta^a} \right)
    &&\text{in } (0, \infty) \times \Omega,
    \\
   \frac{1}{\tau_R} \partial_{t^\ast} \theta_i - \frac{1}{\tau^{\mathrm{diff},\Sigma}} \dv_\Sigma^\ast ( \sum_{j=0}^N d_{ij}^{\Sigma,\ast} \nabla_\Sigma^\ast \theta_j )
    &= \sum_{a} \nu_i^{\Sigma,a} \left( \frac{1}{\tau_a^{\mathrm{react},\vec \theta,f}} k_a^{\Sigma,f, \ast} \vec \theta^{\vec \alpha^{\Sigma,a}} - \frac{1}{\tau_a^{\mathrm{react},\vec \theta,b}} k_a^{\Sigma, b, \ast} \vec \theta^{\vec \beta^{\Sigma,a}} \right)
     \\ &\qquad
     + \left( \frac{1}{\tau_i^\mathrm{ad}} k_i^{\mathrm{ad}, \ast} c_i^\ast \theta_0 - \frac{1}{\tau_i^\mathrm{de}} k_i^{\mathrm{de}, \ast} \theta_i \right)
     &&\text{on } (0, \infty) \times \Sigma,
     \\
    - \frac{1}{\tau^\mathrm{trans}} d_i^\ast \partial_{\vec n^\ast} c_i^\ast
     &= \frac{1}{\tau_i^{\mathrm{ad}}} k_i^{\mathrm{ad}, \ast} c_i^\ast \theta_0 - \frac{1}{\tau_i^{\mathrm{de}}} k_i^{\mathrm{de}, \ast} \theta_i
     &&\text{on } (0, \infty) \times \Sigma.
  \end{align*}
  To prepare the dimensional analysis of the model, denote by
   \begin{align*}
    \tau^\mathrm{react}
     := \tau^\mathrm{react}_\mathrm{slow}
     &= \max \{ \tau_a^{\mathrm{react},f}, \tau_a^{\mathrm{react},b}: \, a = 1, \ldots, m \},
     \\
    \tau^\mathrm{react}_\mathrm{fast}
     &= \min \{ \tau_a^{\mathrm{react},f}, \tau_a^{\mathrm{react},b}: \, a = 1, \ldots, m \}
   \end{align*}
  the characteristic values for the slowest and fastest bulk reaction.
  Analogously, define $\tau^{\mathrm{react},\Sigma} = \tau^{\mathrm{react},\Sigma}_\mathrm{slow} \geq \tau^{\mathrm{react}, \Sigma}_\mathrm{fast} > 0$ and $\tau^\mathrm{sorp} = \tau^\mathrm{sorp}_\mathrm{slow} \geq \tau^\mathrm{sorp}_\mathrm{fast} > 0$. Further, with
   \begin{align*}
    \lambda_a^f
     &= \frac{\tau^\mathrm{react}_\mathrm{slow}}{\tau_a^{\mathrm{react},f}},
     \quad
    \lambda_a^b
     = \frac{\tau^\mathrm{react}_\mathrm{slow}}{\tau_a^{\mathrm{react},b}},
     \quad
    \lambda_a^{\Sigma,f}
     = \frac{\tau^\mathrm{react,\Sigma}_\mathrm{slow}}{\tau_a^{\mathrm{react,\Sigma},f}},
     \\
    \lambda_a^{\Sigma,b}
     &= \frac{\tau^\mathrm{react, \Sigma}_\mathrm{slow}}{\tau_a^{\mathrm{react}, \Sigma,b}},
     \quad
    \lambda_i^\mathrm{ad}
     = \frac{\tau^\mathrm{sorp}_\mathrm{slow}}{\tau_i^\mathrm{ad}},
     \quad
    \lambda_i^\mathrm{de}
     = \frac{\tau^\mathrm{sorp}_\mathrm{slow}}{\tau_i^\mathrm{de}}
     \in [1, \infty)
   \end{align*}
  we may now introduce the following short-hand notation:
   \begin{align*}
    \vec r^\ast(\vec c^\ast)
     &= \sum_{a} \nu_i^a \big( \lambda_a^f k_a^{f,\ast} (\vec c^\ast)^{\vec \alpha^a} - \lambda_a^b k_a^{b,\ast} (\vec c^\ast)^{\vec \beta^a} \big),
     \\
    \vec r^{\Sigma,\ast}(\vec \theta)
     &= \sum_{a} \nu_i^{\Sigma,a} \big( \lambda_a^{\Sigma,f} k_a^{\Sigma,f, \ast} \vec \theta^{\vec \alpha^{\Sigma,a}} - \lambda_a^{\Sigma,b} k_a^{\Sigma, b, \ast} \vec \theta^{\vec \beta^{\Sigma,a}} \big),
     \\
   \vec s^{\Sigma,\ast}(\vec c^\ast, \vec \theta)
    &= \lambda_i^\mathrm{ad} k_i^{\mathrm{ad}, \ast} c_i^\ast \theta_0 - \lambda_i^\mathrm{de} k_i^{\mathrm{de}, \ast} \theta_i.
   \end{align*}
 
 \begin{remark}
  The parameters $\lambda_a^f$, $\lambda_a^{\Sigma,f}$, $\lambda_i^\mathrm{ad}$ may, in general, be very small, corresponding to chemical reactions in the bulk taking place on very dissimilar time scales.
  In that case an additional limit process is possible, which may lead, e.g.\ to some reversible reactions being replaced by an irreversible reaction, or to the negligence of slow reactions compared to faster reactions. Since the focus of this manuscript, however, lies on limit models for fast sorption and fast surface diffusion, we consider only the simpler case, where, e.g.\ $\tau^\mathrm{react,\Sigma}_\mathrm{slow} \ll \tau^\mathrm{sorp}_\mathrm{fast}$.
  I.e.\ we assume that even the slowest surface chemical reactions are faster than the fastest sorption processes.
 \end{remark}
 
 With this notation at hand, the full bulk-surface reaction-diffusion-sorption model can be written in the following condensed dimensionless form.

  \begin{align}
   \frac{1}{\tau_R} \partial_{t^\ast} c_i^\ast - \frac{1}{\tau^\mathrm{diff}} d_i^\ast \Delta^\ast c_i^\ast
    &= \frac{1}{\tau^\mathrm{react}} r_i^\ast(\vec c^\ast)
    &&\text{in } (0, \infty) \times \Omega,
    \label{eqn:dimless_bulk}
    \\
   \frac{1}{\tau_R} \partial_{t^\ast} \theta_i - \frac{1}{\tau^{\mathrm{diff},\Sigma}} \dv_\Sigma^\ast ( \sum_{j=0}^N d_{ij}^{\Sigma,\ast}(\vec \theta) \nabla_\Sigma^\ast \theta_j )
    &= \frac{1}{\tau^{\mathrm{react},\Sigma}} r_i^{\Sigma,\ast}(\vec \theta)
     + \frac{1}{\tau^\mathrm{sorp}} s_i^{\Sigma,\ast}(\vec c^\ast, \vec \theta)
    &&\text{on } (0, \infty) \times \Sigma,
     \label{eqn:dimless_surface}
     \\
    - \frac{1}{\tau^\mathrm{trans}} d_i^\ast \partial_{\vec n^\ast} c_i^\ast
     &= \frac{1}{\tau^\mathrm{sorp}} s_i^{\Sigma,\ast}(\vec c^\ast, \vec \theta)
     &&\text{on } (0, \infty) \times \Sigma.
     \label{eqn:dimless_transmission}
  \end{align}
 Starting from there, several limit cases can be considered.
 
 \begin{enumerate}
  \item
   One-parameter limits, i.e.\ exactly one of the thermodynamic subprocesses (here:\ either surface chemistry, sorption or species transport to/from the surface) is assumed to be much faster than all of the other physico-chemical processes.
  \item
   Two-parameter limits, i.e.\ two thermodynamic processes are considerably faster than the others, e.g.\ fast surface chemistry and sorption.
  \item
   Three-parameter limits, i.e.\ also a third thermodynamic process is much faster than the other, remaining thermodynamic processes.
 \end{enumerate}

 \section{Limit models}

 Within this section, several limit models, each of them corresponding to different ordering of the times scales $\tau^R$, $\tau^{\mathrm{diff}, \Sigma}$, $\tau^{\mathrm{react},\Sigma}$, $\tau^\mathrm{sorp}$ and $\tau^\mathrm{trans}$  are motivated and derived.
 These constitute the most relevant cases for the dynamics on the surface and the transmission condition between bulk and surface.
 First, the cases with only one fast thermodynamic process, corresponding to one of the time scale parameters being very small compared to the others, are investigated. Thereafter, cases of two or three very fast processes are considered as well, possibly with a structural relation between some of the fast processes, e.g.\ one of them being ultra-fast, i.e.\ even being very fast compared to other fast processes.
 
 \subsection{One-parameter limits}
 Starting with the case of exactly one fast physical or chemical process, (at least mathematically) five limit cases can be distinguished:
  \begin{enumerate}
   \item
    fast surface chemistry, characterised by the condition
     \[
      \tau^{\mathrm{react},\Sigma}_\mathrm{slow} \ll \tau^\mathrm{R}, \tau^{\mathrm{diff},\Sigma}, \tau^\mathrm{sorp}_\mathrm{fast};
     \]
   \item
    fast sorption processes, characterised by the condition
     \[
      \tau^\mathrm{sorp}_\mathrm{slow} \ll \tau^\mathrm{R}, \tau^{\mathrm{diff},\Sigma}, \tau^{\mathrm{react}, \Sigma}_\mathrm{fast};
     \]
   \item
    the fast surface diffusion case, characterised by the condition
     \[
      \tau^{\mathrm{diff},\Sigma} \ll \tau^\mathrm{R}, \tau^{\mathrm{react}, \Sigma}_\mathrm{fast}, \tau^\mathrm{sorp}_\mathrm{fast};
     \]
   \item
    the fast accumulation case, characterised by the condition
     \[
      \tau^\mathrm{R} \ll \tau^{\mathrm{diff},\Sigma}, \tau^{\mathrm{react}, \Sigma}_\mathrm{fast}, \tau^\mathrm{sorp}_\mathrm{fast};
     \]
   \item
    the case of fast transmission between surface and bulk, characterised by the condition
     \[
      \tau^\mathrm{trans} \ll \tau^\mathrm{R}, \tau^{\mathrm{diff},\Sigma}, \tau^{\mathrm{react}, \Sigma}_\mathrm{fast}, \tau^{\mathrm{sorp}}_\mathrm{fast}.
     \]
  \end{enumerate}
 
 \subsubsection{Fast surface chemistry}
 First, assume that
  \[
   \tau^{\mathrm{react},\Sigma}_\mathrm{slow} \ll \tau^\mathrm{R}, \tau^{\mathrm{diff},\Sigma}, \tau^\mathrm{sorp}_\mathrm{fast}.
  \]
 This is the most typical case for heterogeneous catalysts and means that, on the surface, the chemical reactions take place much faster than all other physical-chemical processes, in particular bulk and surface diffusion, bulk chemistry and sorption at the surface.
 From a chemical engineering point of view this often is the most desirable case.
  \begin{remark}
   It is possible to consider only some of the chemical reactions on the surface as being fast, leading to fewer (nonlinear) constraints below and additional slow or moderately fast surface reaction terms in the dynamics on $\lin \{\vec e^k\}$ as defined below.
   In this case, in the fast reaction limit, only some of the chemical reactions would be assumed to be infinitely fast whereas other reactions take place on the same time scale as the remaining physical and chemical mechanisms.
  \end{remark}
 The limiting case is (formally) obtained by multiplying the evolutionary equation on the surface \eqref{eqn:dimless_surface} by $\tau^{\mathrm{react},\Sigma}$ and then letting $\tau^{\mathrm{react},\Sigma} \rightarrow 0$, leading formally to the algebraic, quasi-steady nonlinear relation
  \[
   \vec r^{\Sigma, \ast}(\vec \theta) = \vec 0,
  \]
 which due to the definition $r_0^{\Sigma,\ast}(\vec \theta) = - \sum_{i=1}^N r_i^{\Sigma,\ast}(\vec \theta)$ is equivalent to
  \[
   \vec r_\mathrm{red}^{\Sigma,\ast}(\vec \theta)
    = (r_i^{\Sigma,\ast}(\vec \theta))_{i=1,\ldots,N}
    = \vec 0.
  \]
  \begin{remark}[On the condition $\vec r^{\Sigma,\vec \ast}(\vec \theta) = \vec 0$]
   The condition $\vec r^{\Sigma, \ast}(\vec \theta) = \vec 0$ can be interpreted as follows.
    \begin{enumerate}
     \item
      First, consider the case $m^\Sigma = 1$, i.e.\ only one type of chemical reaction $\sum_{i=0}^N \alpha_i^\Sigma A_i^\Sigma \rightleftharpoons \sum_{i=0}^N \beta_i^\Sigma A_i^\Sigma$ takes place on the surface. From the modelling of the chemical reaction rates one has
      \[
       \frac{R^{\Sigma,f}}{R^{\Sigma,b}} = - \exp( \frac{1}{RT} \mathcal{A}^\Sigma )
      \]
      and $R^\Sigma = R^{\Sigma,f} - R^{\Sigma,b} = 0$ if and only if the affinity vanishes: $\mathcal{A}^\Sigma = \sum_{i=0}^N \mu_i^\Sigma \nu_i^\Sigma = 0$.
     \item
      In case of $m^\Sigma \geq 2$ types of chemical reactions $\sum_{i=0}^N \alpha_i^{\Sigma,a} A_i^\Sigma \rightleftharpoons \sum_{i=0}^N \beta_i^{\Sigma,a} A_i^\Sigma$, $a = 1, \ldots, m^\Sigma$, the condition $\vec r^{\Sigma, \ast}(\vec \theta) = \vec 0$ is, in general, \emph{not} equivalent to the condition $\mathcal{A}_a^\Sigma = 0$ for $a = 1, \ldots, m^\Sigma$. Equivalence of $\vec r^{\Sigma, \ast}(\vec \theta) = 0$ to the condition $\mathcal{A}_a^\Sigma = 0$ for $a = 1, \ldots, m^\Sigma$, holds true if and only if the stochiometric vectors $\vec \nu^{\Sigma,a}$, $a  = 1, \ldots, m^\Sigma$, are \emph{linearly independent}. More precisely, $\vec r^{\Sigma, \ast}(\vec \theta) = \vec 0$ holds true if and only if $R_a^\Sigma(\vec \theta) = R_a^{\Sigma,f}(\vec \theta) - R_a^{\Sigma,b}(\vec \theta) = w_a$ for some vector $\vec w \in \ker \big( \left[ \vec \nu^{\Sigma,1} \cdots \vec \nu^{\Sigma^{m^\Sigma}} \right]^\mathsf{T} \big) \subseteq \R^{m^\Sigma}$. Chemical equilibria for which $\vec w \neq \vec 0$ is allowed are called \emph{complex-balanced equilibria}, whereas equilibria with $\vec w = \vec 0$ (i.e.\ $\mathcal{A}_a = 0$, $a = 1, \ldots, m^\Sigma$) are called \emph{detailed-balanced equilibria}.
     \item
      In any case, $\vec r^{\Sigma,\ast}(\vec \theta) \in \lin \{\vec \nu^{\Sigma,a}: \, a = 1, \ldots, m^\Sigma\}$ and, hence, the static condition may replace the dynamic equation \eqref{eqn:dimless_surface} only on $\lin \{\vec \nu^{\Sigma,a}: \, a = 1, \ldots, m^\Sigma\}$, while on $\{\vec \nu^{\Sigma,a}: a = 1, \ldots, m^\Sigma \}^\bot$ a dynamic condition still remains in the fast surface chemistry limit.
    \end{enumerate}
  \end{remark}
  
  \begin{assumption}[Detailed-balanced equilibria]
  \label{assmpt:detailled_balance}
   Throughout this manuscript, we assume that all equilibria are detailed-balanced, i.e.\
    \[
     \vec \nu^{\Sigma,1}, \ldots, \vec \nu^{\Sigma,m^\Sigma} \quad \text{are linearly independent.}
    \]
   (As by definition $\nu_0^{\Sigma,a} = - \sum_{i=1}^N \nu_i^{\Sigma,a}$, this condition is equivalent to the statement that $\vec \nu_\mathrm{red}^{\Sigma,1}, \ldots, \vec \nu_\mathrm{red}^{\Sigma,m^\Sigma}$ with $\vec \nu_\mathrm{red}^{\Sigma,a} := (\nu_i^{\Sigma,a})_{i=1,\ldots,N} \in \R^N$ are linearly independent.)
  \end{assumption}
 For $\big( \lin \{\vec \nu_\mathrm{red}^{\Sigma,a}: a\} \big)^\bot \subseteq \R^N$ one then finds an (orthonormal) basis $\{\vec e^k: \, k = 1, \ldots, n^\Sigma\} \subseteq \R^N$. Considering the inner product of \eqref{eqn:dimless_surface} with $\vec e^k$, $k = 1, \ldots, n^\Sigma$, one obtains the evolution equations
  \[
   \frac{1}{\tau^\mathrm{R}} \vec e^k \cdot \partial_{t^\ast} \vec \theta_\mathrm{red} - \frac{1}{\tau^{\mathrm{diff},\Sigma}} \vec e^k \cdot \dv_\Sigma (\bb D^{\Sigma,\ast} \vec \nabla_\Sigma \vec \theta) = \frac{1}{\tau^\mathrm{sorp}} \vec e^k \cdot \vec s^{\Sigma,\ast}(\vec c^\ast, \vec \theta),
   \quad
   k = 1, \ldots, n^\Sigma,
  \]
 where $\bb D^{\Sigma,\ast} = (d_{ij}^{\Sigma,\ast})_{i,j} \in \R^{N \times (1+N)}$.
 The resulting reduced model is given by a coupled bulk-surface reaction-sorption-system with a nonlinear constraint on the surface occupancy numbers $\vec \theta$:
  \begin{align*}
   \frac{1}{\tau^\mathrm{R}} \partial_{t^\ast} \vec c^\ast - \frac{1}{\tau^\mathrm{diff}} \bb D^\ast \Delta^\ast \vec c^\ast
    &= \frac{1}{\tau^\mathrm{react}} r_i^\ast(\vec c^\ast),
    &&t \geq 0, \, \vec z \in \Sigma,
    \\
   \frac{1}{\tau^\mathrm{R}} \vec e^k \cdot \partial_{t^\ast} \vec \theta_\mathrm{red} - \vec e^k \cdot \dv_\Sigma (\bb D^{\Sigma,\ast} \nabla_\Sigma^\ast \vec \theta)
    &= \frac{1}{\tau^\mathrm{sorp}} \vec e^k \cdot \vec s^{\Sigma,\ast}(\vec c^\ast, \vec \theta),
    &&k = 1, \ldots, n^\Sigma, \, t \geq 0, \, \vec z \in \Sigma
    \\
   - \frac{1}{\tau^\mathrm{trans}} \bb D^\ast \partial_{\vec n^\ast} \vec c^\ast
    &= \frac{1}{\tau^\mathrm{sorp}} \vec s^{\Sigma,\ast}(\vec c^\ast, \vec \theta),
    &&t \geq 0, \, \vec z \in \Sigma,
    \\
   \vec r^{\Sigma, \ast}(\vec \theta)
    &= \vec 0,
    &&t \geq 0, \, \vec z \in \Sigma,
  \end{align*}
 using the notation $\bb D^\ast = \diag (d_i^\ast)_i \in \R^{N \times N}$.
 Returning to the variables $(\vec c, \vec c^\Sigma)$ instead of $(\vec c^\ast, \vec \theta_\mathrm{red})$ this model reads as
  \begin{align*}
   \partial_t \vec c - \bb D \Delta \vec c
    &= r_i(\vec c),
    &&t \geq 0, \, \vec z \in \Sigma,
    \\
   \vec e^k \cdot \partial_t \vec c^\Sigma - \vec e^k \cdot \dv_\Sigma(\bb D^\Sigma \nabla_\Sigma \vec c_\mathrm{ext}^\Sigma)
    &= \vec e^k \cdot \vec s^\Sigma(\vec c, \vec \theta_\mathrm{red}),
    &&k = 1, \ldots, n^\Sigma, \, t \geq 0, \, \vec z \in \Sigma
    \\
   - \bb D \partial_{\vec n} \vec c
    &= \vec s^\Sigma(\vec c, \vec \theta),
    &&t \geq 0, \, \vec z \in \Sigma,
    \\
   \vec r^{\Sigma}(\vec \theta)
    &= \vec 0,
    &&t \geq 0, \, \vec z \in \Sigma
  \end{align*}
 where $\vec c_\mathrm{ext}^\Sigma = (c_0^\Sigma, \vec c^\Sigma)^\mathsf{T}$, $\bb D = \diag (d_i)_i \in \R^{N \times N}$ and $\bb D^\Sigma = (d_{ij}^\Sigma)_{i,j} \in \R^{N \times (1+N)}$.
  \begin{remark}
   When letting $\tau^{\mathrm{react},\Sigma} \rightarrow 0$ and $\vec r^{\Sigma, \ast}(\vec \theta) \rightarrow \vec 0$ at the same time, the term $\vec r^{\Sigma}(\vec \theta) = \frac{1}{\tau^{\mathrm{react},\Sigma}} \vec r^{\Sigma, \ast}(\vec \theta)$ is not well-defined in the limit.
   On the other hand, the limiting process $\tau^{\mathrm{react},\Sigma} \rightarrow 0$ is just an idealised version of fast surface chemistry.
   Therefore, one might handle $\tau^{\mathrm{react},\Sigma} \ll 1$ and $\vec r^{\Sigma, \ast} \ll 1$ very small, but yet finite and non-zero in general, so that the reaction rate vector $\vec r^{\Sigma}(\vec \theta) = \frac{1}{\tau^{\mathrm{react},\Sigma}} \vec r^{\Sigma, \ast}(\vec \theta)$ still is well-defined.
   The dimensional analysis then just serves as a motivation, why the terms $\vec r^{\Sigma}(\vec \theta)$ appearing in the evolution equation on the surface, may be replaced by a nonlinear static relation, thus, restricting the dynamics on the surface to the subspace $\lin \{\vec e^k\}$.
  \end{remark}
 From the above reasoning it is clear how a general fast surface chemistry limit looks like:
  \begin{align*}
   \partial_t \vec c + \dv \bb J
    &= \vec r(\vec c),
    &&t \geq 0, \, \vec z \in \Sigma,
    \\
   \vec e^k \cdot \partial_t \vec c^\Sigma + \vec e^k \cdot \dv_\Sigma \bb J^\Sigma
    &= \vec e^k \cdot \vec s^\Sigma(\vec c, \vec \theta),
    &&k = 1, \ldots, n^\Sigma, \, t \geq 0, \, \vec z \in \Sigma,
    \\
   \bb J \cdot \vec n
    &= \vec s^\Sigma(\vec c, \vec \theta),
    &&t \geq 0, \, \vec z \in \Sigma,
    \\
   \vec r^{\Sigma}(\vec \theta)
    &= \vec 0,
    &&t \geq 0, \, \vec z \in \Sigma.
  \end{align*}
 
 \subsubsection{Fast sorption}
 \label{subsubsec:fast_sorption}
 Next, assume that the sorption process at the surface is considerably faster than all other physical-chemical processes, including the surface chemistry.
 I.e., for the characteristic parameters one has
   \[
    \tau^\mathrm{sorp}_\mathrm{slow} \ll \tau^\mathrm{R}, \tau^{\mathrm{diff}, \Sigma}, \tau^{\mathrm{react}, \Sigma}_\mathrm{fast}.
   \]
  The fast sorption limit is then obtained by multiplying equation \eqref{eqn:dimless_surface} by the characteristic time for the sorption processes $\tau^\mathrm{sorp} > 0$ and letting formally $\tau^\mathrm{sorp} \rightarrow 0$, leading to the quasi-static relations
   \[
    s_i^{\Sigma,\ast}(\vec c^\ast, \vec \theta)
     = 0,
     \quad
     i = 1, \ldots, N, \, t \geq 0, \, \vec z \in \Sigma.
   \]
  Since $\tau^\mathrm{sorp} \rightarrow 0$ and $s^{\Sigma,\ast}(\vec c^\ast, \vec \theta) \rightarrow 0$, the relation $- \frac{1}{\tau^\mathrm{trans}} \partial_{\vec n^\ast} c_i^\ast = \frac{1}{\tau^\mathrm{sorp}} s_i^{\Sigma,\ast}(\vec c^\ast, \vec \theta)$ cannot be used as a boundary condition in the model anymore, but $\frac{1}{\tau^\mathrm{sorp}} s_i^{\Sigma,\ast}(\vec c^\ast, \vec \theta)$ has to be replaced by $- \frac{1}{\tau^\mathrm{trans}} d_i^\ast \partial_{\vec n^\ast} c_i^\ast|_\Sigma$ in the dynamics of the occupancy numbers $\vec \theta$, so that the reduced limit model reads as
  \begin{align*}
   \frac{1}{\tau^\mathrm{R}} \partial_{t^\ast} c_i^\ast - \frac{1}{\tau^\mathrm{diff}} d_i^\ast \Delta^\ast c_i^\ast
    &= \frac{1}{\tau^\mathrm{react}} r_i^\ast(\vec c^\ast)
    &&\text{in } (0,\infty) \times \Omega,
    \\
   \frac{1}{\tau^\mathrm{R}} \partial_{t^\ast} \theta_i - \frac{1}{\tau^{\mathrm{diff},\Sigma}} \dv_\Sigma^\ast ( \sum_{j=0}^N d_{ij}^{\Sigma,\ast}(\vec \theta) \nabla_\Sigma^\ast \theta_j )
    &= \frac{1}{\tau^{\mathrm{react}, \Sigma}} r_i^{\Sigma,\ast}(\vec \theta)
     - \frac{1}{\tau^\mathrm{trans}} d_i^\ast \partial_{\vec n^\ast} c_i^\ast
    &&\text{on } (0,\infty) \times \Sigma,
     \\
    \vec s^{\Sigma,\ast}(\vec c^\ast, \vec \theta)
     &= 0,
     &&\text{on } (0, \infty) \times \Sigma.
  \end{align*}
 Returning to $\vec c$ and $\vec c^\Sigma$ instead of $\vec c^\ast$ and $\vec \theta$, this limit model reads as
  \begin{align*}
   \partial_t \vec c - \bb D \Delta \vec c
    &= \vec r(\vec c),
    &&t \geq 0, \, \vec z \in \Omega,
    \\
   \partial_t \vec c^\Sigma - \dv_\Sigma (\bb D^\Sigma(\vec c^\Sigma) \nabla_\Sigma \vec c_\mathrm{ext}^\Sigma))
    &= \vec r^\Sigma(\vec c^\Sigma) - \bb D \partial_{\vec n} \vec c,
    &&t \geq 0, \, \vec z \in \Sigma,
    \\
   \vec s^\Sigma(\vec c, \vec c^\Sigma)
    &= 0,
    &&t \geq 0, \, \vec z \in \Sigma.
  \end{align*}
 Alternatively, in the more general form for generic reaction, sorption and bulk and surface diffusion models,
  \begin{align*}
   \partial_t \vec c + \dv \bb J
    &= \vec r(\vec c),
    &&t \geq 0, \, \vec z \in \Omega
    \\
   \partial_t \vec c^\Sigma + \dv \bb J^\Sigma
    &= \vec r^\Sigma(\vec c^\Sigma) + \bb J \cdot \vec n,
    &&t \geq 0, \, \vec z \in \Sigma
    \\
   \vec s^\Sigma(\vec c, \vec c^\Sigma)
    &= 0,
    &&t \geq 0, \, \vec z \in \Sigma.
  \end{align*}
 The latter seems to be an appropriate model when considering fast sorption limits for more general reaction-diffusion-sorption models, say Maxwell-Stefan diffusion in the bulk, or more general sorption, reaction or diffusion models on the surface, e.g.\ more general models for the chemical potentials in the bulk and on the surface.
 In fact, the limit considerable simplifies the situation for generic surface chemical potentials, cf.\ the following remark.
 
 \begin{remark}[On the condition $\vec s^\Sigma(\vec c|_\Sigma, \vec c^\Sigma) = \vec 0$]
  For the fast surface chemistry limit it has been demonstrated that the nonlinear equilibrium condition $\vec r^\Sigma(\vec c^\Sigma) = \vec 0$ (under mild assumptions on the structure of surface reactions) is equivalent to $R_a(\vec c^\Sigma) = 0$ for all surface chemical reactions $a$, and the latter condition can be expressed as $\mathcal{A}_a^\Sigma = 0$, i.e., vanishing affinity for all surface chemical reactions.
  As the sorption processes at the surface are modelled analogously to a chemical reaction $A_i \rightleftharpoons A_i^\Sigma$, a similar observation can be made for the sorption equilibrium $\vec s^\Sigma(\vec c|_\Sigma, \vec c^\Sigma) = \vec 0$, namely
   \[
    \vec s^\Sigma(\vec c|_\Sigma, \vec c^\Sigma)
     = \vec 0
     \quad
     \Leftrightarrow
     \quad
    \vec \mu|_\Sigma = \vec {\tilde \mu^\Sigma}
     \text{on } \Sigma,
   \]
  where $\vec \mu = (\mu_i)_i$ and $\vec {\tilde \mu^\Sigma} = (\tilde \mu_i^\Sigma)_i = (\mu_i^\Sigma - \mu_0^\Sigma)$ are the vectors of bulk resp.\ surface chemical potentials.
 \end{remark}
  This observation is actually \emph{independent} of the particular choice for the bulk and surface chemical potentials $\mu_i$ and $\mu_i^\Sigma$, but is implied by the detailed-balance condition on the adsorption and desorption velocities $R_i^\mathrm{ad}$ and $R_i^\mathrm{de}$ for the sorption process and by the entropy production due to sorption.
 
 \subsubsection{Fast surface diffusion}
  The fast surface diffusion case is characterised by the condition
     \[
      \tau^{\mathrm{diff}, \Sigma} \ll \tau^\mathrm{R}, \tau^{\mathrm{react}, \Sigma}, \tau^\mathrm{sorp}.
     \]
   Multiplying the dynamic equation \eqref{eqn:dimless_surface} for the surface occupancy numbers by $\tau^{\mathrm{diff}, \Sigma}$ and taking the formal limit $\tau^{\mathrm{diff},\Sigma} \rightarrow 0$ leads to the constraint
    \[
     \dv_\Sigma^\ast (\sum_{j=0}^N d_{ij}^{\Sigma,\ast}(\vec \theta) \nabla_\Sigma^\ast \theta_j)
      = 0
      \quad
      \text{on } \Sigma.
    \]
   From here, in the situation of standard Fickian diffusion $d_{ij}^\Sigma = \delta_{ij} d_i^\Sigma$ for strictly positive $d_i^\Sigma > 0$, it would follow that $\nabla \vec \theta = \bb 0$, hence $\vec \theta(t, \vec z) = \vec \theta(t)$ would not depend on the spatial position $\vec z \in \Sigma$.
   For Fick-Onsager diffusion, however, this deduction is not possible, and in fact from the constraint $\sum_{i=0}^N \vec j_i^{\Sigma,\mathrm{site}} = \vec 0$ (thus $\sum_{j=0}^N d_{ij}^\Sigma = 0$), it follows that $\vec e = (1, \ldots, 1)^\mathsf{T} \in \R^{1+N}$ is a non-trivial element of $\ker (\bb D_\mathrm{ext}^{\Sigma,\ast})$, and if one demands that $d_{ij}^\Sigma \leq 0$ for $i \neq j$ and $d_{ii}^\Sigma > 0$, by Perron-Frobenius theory $\ker (\bb D_\mathrm{ext}^{\Sigma,\ast}) = \lin \{\vec e\}$.
   Thus, in that case the limit model reduces the surface evolution of $\vec \theta$ to an evolution on the constrainted space $\lin \{\vec e\}$, and by the definition $\theta_0 = 1 - \sum_{i=1}^N \theta_i$ this means that $\vec \theta(t, \vec z) = (1 - \theta_0(t, \vec z)) \frac{\vec e}{N}$ in the limit model, which then reads
    \begin{align*}
     \frac{1}{\tau^\mathrm{R}} \partial_{t^\ast} \theta_0
      &= - \frac{1}{\tau^\mathrm{R}} \sum_{i=1}^N \partial_{t^\ast} \theta_i
      = - \sum_{i=1}^N \big( \frac{1}{\tau^{\mathrm{react},\Sigma}} r_i^{\Sigma,\ast}(\vec \theta) + \frac{1}{\tau^{\mathrm{sorp}}} s_i^{\Sigma,\ast}(c_i|_\Sigma, \vec \theta) \big)
      \\      
      &= \frac{1}{\tau^{\mathrm{react},\Sigma}} r_0^{\Sigma,\ast}(\tfrac{1-\theta_0}{N} \vec e) + \frac{1}{\tau^{\mathrm{sorp}}} s_0^{\Sigma,\ast}(\vec c|_\Sigma, \tfrac{1-\theta_0}{N} \vec e)
    \end{align*}
   with $r_0^{\Sigma,\ast} := - \sum_{i=1}^N r_i^{\Sigma,\ast}$ and $s_0^{\Sigma,\ast} := - \sum_{i=1}^N s_i^{\Sigma,\ast}$.
   The dynamics on the surface then reduces to an (parameter $\vec z \in \Sigma$ dependent) ODE evolution equation for the vacancies and the limit model is
    \begin{align*}
     \frac{1}{\tau^\mathrm{R}} \partial_{t^\ast} c_i^\ast
      - \frac{1}{\tau^{\mathrm{diff},\Sigma}} d_i^\ast \Delta^\ast
      &= \frac{1}{\tau^\mathrm{react}} r_i^\ast(\vec c^\ast)
      &&\text{in } (0, \infty) \times \Omega,
      \\
     \frac{1}{\tau^\mathrm{R}} \partial_{t^\ast} \theta_0
      &= \frac{1}{\tau^{\mathrm{react},\Sigma}} r_0^{\Sigma,\ast}(\tfrac{1-\theta_0}{N} \vec e) + \frac{1}{\tau^\mathrm{sorp}} s_0^{\Sigma,\ast}(\vec c^\ast|_\Sigma, \tfrac{1-\theta_0}{N} \vec e)
      &&\text{on } (0, \infty) \times \Sigma,
      \\
     - \frac{1}{\tau^\mathrm{trans}} d_i^\ast \partial_{\vec n^\ast} c_i^\ast|_\Sigma
      &= \frac{1}{\tau^\mathrm{sorp}} s_i^{\Sigma,\ast}(c_i^\ast|_\Sigma,\frac{1-\theta_0}{N} \vec e)
      &&\text{on } (0, \infty) \times \Sigma.
    \end{align*}
   This is a PDE with local, dynamic ODE boundary conditions, and -- returning to the variables set $(\vec c, c_0^\Sigma)$ -- may be formulated as
    \begin{align*}
     \partial_t \vec c
      - \dv (\bb D \nabla \vec c)
      &= \vec r(\vec c)
      &&\text{in } (0, \infty) \times \Omega,
      \\
     \partial_t c_0^\Sigma
      &= r_0^\Sigma(\tfrac{c_S^\Sigma - c_0^\Sigma}{N} \vec e) + s_0^\Sigma(\vec c|_\Sigma, \tfrac{c_S^\Sigma - c_0^\Sigma}{N} \vec e)
      &&\text{on } (0, \infty) \times \Sigma,
      \\
     - \bb D \partial_{\vec n} \vec c|_\Sigma
      &= \vec s^\Sigma(\vec c|_\Sigma, \tfrac{c_S^\Sigma - c_0^\Sigma}{N} \vec e)
      &&\text{on } (0, \infty) \times \Sigma.
    \end{align*}
    
 \subsubsection{Fast surface accumulation}
 In the fast surface accumulation case it holds that
     \[
      \tau^\mathrm{R} \ll \tau^{\mathrm{diff},\Sigma}, \tau^{\mathrm{react},\Sigma}_\mathrm{fast}, \tau^\mathrm{sorp}_\mathrm{fast}.
     \]
    This possibly is the mathematically most delicate and challenging case, as multiplying the dimensionless formulation of the surface dynamics by $\tau^\mathrm{acc}$ and then taking the formal limit $\tau^ \mathrm{acc} \rightarrow 0$ leads to
     \[
      \partial_{t^\ast} \vec \theta
       = \vec 0.
     \]
    However, when returning to the original equation where the factor $\tau^\mathrm{acc} \partial_{t^\ast} \vec \theta$ appears, one looses \emph{any} information on the dynamic behaviour of $\vec \theta$.
    Therefore, one should rather interpret the limit $\tau^\mathrm{acc} \rightarrow 0$ as follows. For any given time $t_0 \geq 0$ one fixes the values $\vec c(t_0,\vec z)$ for $\vec z \in \Sigma$, and then considers the following system of evolution equations on $\Sigma$:
     \begin{align*}
      \partial_{t^\ast} \theta_i(t_0 + t^\ast, \vec z)
       &= \tau^\mathrm{R} \left( \frac{1}{\tau^{\mathrm{diff},\Sigma}} \dv_\Sigma^\ast (d_{ij}^{\Sigma,\ast}(\vec \theta(t_0 + t^\ast)) \nabla_\Sigma^\ast \theta_j(t_0 + t^\ast, \vec z))
        \right. \\ &\qquad \left.
        + \frac{1}{\tau^{\mathrm{react},\Sigma}} r_i^{\Sigma,\ast}(\vec \theta(t_0 + t^\ast, \vec z) + \frac{1}{\tau^\mathrm{sorp}} s_i^{\Sigma,\ast}(\vec c(t, \vec z), \vec \theta(t_0 + t^\ast, \vec z)) \right),
        \quad
        i = 0, 1, \ldots, N.
     \end{align*}
    This system of PDEs on the surface may be solved for any $\tau^\mathrm{R} > 0$, if at least it can be solved for one particular $\tau^\mathrm{R} > 0$, and in this case the corresponding solutions $u(\cdot; t_0, \tau^\mathrm{R})$ are related via
     \[
      u(t_0 + t^\ast; t_0, \alpha \tau^\mathrm{R})
       = u(t_0 + \tfrac{1}{\alpha} t^\ast; t_0, \tau^\mathrm{R}).
     \]
    Taking $\tau^\mathrm{R} \rightarrow 0$ is then equivalent to considering $\alpha \rightarrow 0$ for fixed $\tau^\mathrm{R}$. 
    One obtains
     \[
      u(t_0 + t^\ast; t_0 \alpha \tau^\mathrm{R})
       = u(t_0 + \tfrac{1}{\alpha} t^\ast; t_0, \tau^\mathrm{R})
       \rightarrow u_\infty(t_0) := \lim_{s^\ast \rightarrow \infty} u(t_0 + s^\ast; t_0, \tau^\mathrm{R})
       \quad
       \text{as } \alpha \rightarrow 0,
     \]
    if this limit exists. Hence, the fast accumulation limit can be formulated under the following premisses:
    Assume that $d_{ij}^{\Sigma,\ast}$, $\vec r^{\Sigma, \ast}$ and $\vec s^{\Sigma, \ast}$ are such that for every given $\tau^{\mathrm{diff},\Sigma}$, $\tau^{\mathrm{react}, \Sigma}$ and $\tau^\mathrm{sorp} > 0$ as well as $\vec c: \Sigma \rightarrow \R^N$ regular enough, and every $\vec \theta^0 \in \R^N$, the nonlinear Cauchy problem
     \begin{align*}
      \partial_t \theta_i(t, \vec z)
       &= \frac{1}{\tau^\mathrm{diff}} \dv_\Sigma^\ast (\sum_{j=0}^N d_{ij}^{\Sigma,\ast}(\vec \theta(t)) \nabla_\Sigma \theta_j(t, \vec z))
        \\ &\quad
        + \frac{1}{\tau^\mathrm{react}} r_i^{\Sigma,\ast}(\vec \theta(t, \vec z)) + \frac{1}{\tau^\mathrm{sorp}} s_i^{\Sigma,\ast}(\vec c(\vec z), \vec \theta(t, \vec z))
       &&\text{on } (0, \infty) \times \Sigma,
       \\
      \vec \theta(0,\vec z)
       &= \vec \theta^0(\vec z)
       &&\text{on } \Sigma
     \end{align*}
    has a unique solution $\vec \theta: \R_+ \times \Sigma \rightarrow \{\vec v \in \R^{1+N}: \, \sum_{i=0}^N v_i = 1 \} $ which for $t \rightarrow \infty$ converges to some $\mathcal{P} \vec c := \lim_{t \rightarrow \infty} \vec \theta(t; \vec c, \vec \theta^0)$ which is \emph{independent of $\vec \theta^0$}.
    Then the fast accumulation limit problem can be formulated as
     \begin{align*}
      \partial_t \vec c - \bb D \Delta \vec c
       &= \vec r(\vec c)
       &&\text{in } (0, \infty) \times \Omega,
       \\
      - \bb D \partial_{\vec n} \vec c
       &= \vec s^\Sigma(\vec c, \mathcal{P} \vec c)
       &&\text{on } (0, \infty) \times \Sigma.
     \end{align*}
    In the more general form suitable for generic diffusive flux, bulk chemistry and sorption models one obtains analogously
     \begin{align*}
      \partial_t \vec c + \dv \bb J
       &= \vec r(\vec c)
       &&\text{in } (0, \infty) \times \Omega,
       \\
      \bb J \cdot \vec n
       &= \vec s^\Sigma(\vec c, \mathcal{P} \vec c)
       &&\text{on } (0, \infty) \times \Sigma
     \end{align*}
    if, e.g.\ $\mathcal{P} \vec c$ is a (unique) global attractor for the system of PDE's
     \[
      \partial_t \theta_i(t, \vec z)
       = \frac{1}{\tau^{\mathrm{diff},\Sigma}} \dv_\Sigma^\ast (\sum_{j=0}^N d_{ij}^{\Sigma,\ast}(\vec \theta(t)) \nabla \theta_j(t, \vec z)) + \frac{1}{\tau^{\mathrm{react},\Sigma}} r_i^{\Sigma,\ast}(\vec \theta(t, \vec z)) + \frac{1}{\tau^\mathrm{sorp}} s_i^{\Sigma,\ast}(\vec c(\vec z), \vec \theta(t, \vec z))
     \]
 for $t \geq 0$, $\vec z \in \Sigma$, and $i = 0, 1, \ldots, N$, resp.\ in more abstract form
     \[
      \partial_t c^\Sigma + \dv_\Sigma \bb J^\Sigma
       = r^\Sigma(\vec c^\Sigma) + \vec s^\Sigma(\vec c, \vec c^\Sigma),
       \quad
       t \geq 0, \, \vec z \in \Sigma.
     \]
 Similar topics have been covered in the work \cite{Bot01a} of the second author in a much simpler setting; where instantaneous limit models for fast irreversible reactions have been considered.
 Note that the global attractor $\vec \theta^\infty = \mathcal{P} \vec c$ (if it exists), satisfies the steady-state condition
  \[
   \frac{1}{\tau^{\mathrm{diff},\Sigma}} \dv_\Sigma^\ast (\sum_{j=0}^N d_{ij}^{\Sigma,\ast}(\vec \theta^\infty(\vec z)) \nabla \theta_j^\infty(\vec z)) + \frac{1}{\tau^{\mathrm{react},\Sigma}} r_i^{\Sigma,\ast}(\vec \theta^\infty(\vec z)) + \frac{1}{\tau^\mathrm{sorp}} s_i^{\Sigma,\ast}(\vec c^\ast(\vec z), \vec \theta^\infty(\vec z))
    = 0,
  \] 
 for all $\vec z \in \Sigma$ and $i = 0, 1, \ldots, N$.
 Depending on whether the characteristic parameters $\tau^{\mathrm{diff},\Sigma}$, $\tau^{\mathrm{react},\Sigma}$ and $\tau^\mathrm{sorp} > 0$ are on the same time scale or not, this relation may also serve as a starting point for further reduction of the model.
 
 \subsubsection{Fast transmission between bulk phase and surface}
 \label{subsubsec:fast_transmission}

 Another limit model which can be considered formally, is the limiting case for
  \[
   \tau^\mathrm{trans} \ll \tau^\mathrm{R}, \tau^{\mathrm{diff},\Sigma}, \tau^{\mathrm{react},\Sigma}_\mathrm{fast}, \tau^\mathrm{sorp}_\mathrm{fast}.
  \] 
 However, the limit model which would result has the abstract form
  \begin{align*}
   \partial_t \vec c + \dv \bb J
    &= \vec r(\vec c),
    &&t \geq 0, \, \vec z \in \Omega,
    \\
   \partial_t \vec c^\Sigma + \dv_\Sigma \bb J^\Sigma
    &= \vec r^\Sigma(\vec c^\Sigma) + \vec s^\Sigma(\vec c, \vec c^\Sigma),
    &&t \geq 0, \, \vec z \in \Sigma,
    \\
   \bb J \cdot \vec n
    &= 0,
    &&t \geq 0, \, \vec z \in \Sigma,
  \end{align*}
 so the bulk dynamics would be completely decoupled from the surface dynamics, in the sense that the surface concentrations to not influence the reaction-diffusion system in the bulk at all.
 In particular, this limit case cannot be thermodynamically consistent, it even does not obey to mass conservation.
 \newline
 Our interpretation of this phenomenon is the following: The thermodynamic inconsistency (violation of principle of mass conservation) indicates that taking the fast transmission limit independently of other limits, is not allowed. In fact, the transmission and sorption processes are closely related, hence these sub-processes should take place on the 
 same order of magnitude, i.e.\ $\tau^\mathrm{trans} = \lambda \tau^\mathrm{sorp}$ for some parameter $\lambda > 0$. This can be seen as a motivation for the three-parameter limit considered in subsection \ref{subsec:three_parameter_limit}.
 In that sense, $\tau_\mathrm{trans}$ is \emph{not independent} of the other characteristic parameters, in particular $\tau^\mathrm{sorp}$.
 Actually, this comes without surprise since originally the sorption rates $s_i(\vec c, \vec c^\Sigma)$ are just defined as the outer normal flux $- d_i \partial_{\vec n} c_i$ for species $A_i$ at the boundary $\Sigma = \partial \Omega$.

 \subsection{Two-parameter limits}
 In the previous subsection, several one-parameter limits have been considered, each of them corresponding to a different thermodynamic subprocess which is assumed to take place very fast compared to all other subprocesses.
 Quite typically, however, not only one, but several of these thermodynamic subprocesses take are very fast.
 For this reason, the case of two-parameter limits will be investigated next, where two of the thermodynamic subprocesses are assumed to be much faster than all the other thermodynamic subprocesses.
 Here, the focus lies on the fast sorption and fast surface chemistry limit, and it will also be discussed, how -- if at all -- a hierarchy between the speeds of these two thermodynamic mechanisms does effect the resulting limit model.
 Hence, the following three limit cases will be studied:
  \begin{enumerate}
   \item
    ultra-fast sorption and fast surface chemistry, i.e.\
     \[
      \tau^\mathrm{sorp}_\mathrm{slow} \ll \tau^{\mathrm{react},\Sigma}_\mathrm{fast} \leq \tau^{\mathrm{react}, \Sigma}_\mathrm{slow} \ll \tau^\mathrm{R}, \tau^{\mathrm{diff},\Sigma}, \tau^\mathrm{trans};
     \]
   \item
    fast sorption and ultra-fast surface chemistry, i.e.\
     \[
      \tau^{\mathrm{react},\Sigma}_\mathrm{slow} \ll \tau^\mathrm{sorp}_\mathrm{fast} \leq \tau^\mathrm{sorp}_\mathrm{fast} \ll \tau^\mathrm{R}, \tau^{\mathrm{diff},\Sigma}, \tau^\mathrm{trans};
     \]
   \item
    fast sorption and equivalently fast surface chemistry, i.e.\
     \[
      \tau^{\mathrm{react},\Sigma}_\mathrm{slow} = \lambda \tau^\mathrm{sorp}_\mathrm{slow} \ll \tau^\mathrm{R}, \tau^{\mathrm{diff},\Sigma}, \tau^\mathrm{trans},
      \quad
      \text{for some fixed } \lambda > 0.
     \]
  \end{enumerate}
 
 \subsubsection{Fast surface chemistry, ultra-fast sorption}
 In this model one first takes the formal limit $\tau^\mathrm{sorp} \rightarrow 0$, and in the resulting fast sorption limit model, i.e.\
  \begin{align*}
   \frac{1}{\tau^\mathrm{R}} \partial_{t^\ast} c_i^\ast - \frac{1}{\tau^\mathrm{diff}} d_i^\ast \Delta^\ast c_i^\ast
    &= \frac{1}{\tau^\mathrm{react}} r_i^\ast(\vec c^\ast)
    &&\text{in } (0, \infty) \times \Omega,
    \\
   \frac{1}{\tau^\mathrm{R}} \partial_{t^\ast} \theta_i - \frac{1}{\tau^{\mathrm{diff},\Sigma}} \dv_\Sigma^\ast ( \sum_{j=0}^N d_{ij}^{\Sigma,\ast}(\vec \theta^\ast) \nabla_\Sigma^\ast \theta_j )
    &= \frac{1}{\tau^{\mathrm{react},\Sigma}} r_i^{\Sigma,\ast}(\vec \theta)
     - \frac{1}{\tau^\mathrm{trans}} d_i^\ast \partial_{\vec n^\ast} c_i^\ast
    &&\text{on } (0, \infty) \times \Sigma,
    \\
   s_i^{\Sigma,\ast}(\vec c^\ast, \vec \theta)
    &= 0
    &&\text{on } (0, \infty) \times \Sigma,
  \end{align*}
 additionally considers the formal limit $\tau^{\mathrm{react},\Sigma} \rightarrow 0$.
 This reduces the evolutionary PDE for $\vec \theta$ to the quasi-static relation
  \[
   \vec r^{\Sigma, \ast}(\vec \theta)
    = \vec 0,
    \quad
    \text{on } (0, \infty) \times \Sigma
  \]
 which, as for the one-parameter fast surface chemistry limit, is a condition on the part of $\vec \theta_\mathrm{red}$ lying in the linear span of $\{\vec \nu^{\Sigma,a}: a\}$.
 On its orthogonal complement $\lin \{\vec e^k: \, k = 1, \ldots, n^\Sigma\}$, a dynamic PDE remains, so that the resulting two-parameter limit model reads as
  \begin{align*}
   \frac{1}{\tau^\mathrm{R}} \partial_{t^\ast} \vec c^\ast - \frac{1}{\tau^\mathrm{diff}} \bb D^\ast \Delta^\ast \vec c^\ast
    &= \frac{1}{\tau^\mathrm{react}} \vec r^\ast(\vec c^\ast)
    &&\text{on } (0, \infty) \times \Omega,
    \\
   \frac{1}{\tau^R} \vec e^k \cdot \partial_{t^\ast} \vec \theta_\mathrm{red} - \frac{1}{\tau^{\mathrm{diff},\Sigma}} \vec e^k \cdot \dv_\Sigma^\ast (\bb D^{\Sigma,\ast} \nabla_\Sigma^\ast \vec \theta )
    &= - \frac{1}{\tau^\mathrm{trans}} \vec e^k \cdot \bb D^{\ast} \partial_{\vec n^\ast} \vec c^\ast
    &&\text{in } (0, \infty) \times \Sigma,
     \\
    c_R k_i^{\mathrm{ad}, \ast} c_i^\ast \theta_0 - k_i^{\mathrm{de}, \ast} \theta_i
     &= 0,
    &&k = 1, \ldots, n^\Sigma, \, \text{in } (0, \infty) \times \Sigma,
     \\
    \vec r^{\Sigma, \ast}(\vec \theta)
     &= \vec 0
    &&\text{on } (0, \infty) \times \Sigma.
  \end{align*}
 This limit model will later be compared with the limit model for the other two cases.
  
 \subsubsection{Ultra-fast surface chemistry, fast sorption}
 For this situation one starts the other way round, i.e.\ with the fast surface chemistry model
  \begin{align*}
   \frac{1}{\tau^\mathrm{R}} \partial_{t^\ast} \vec c^\ast - \frac{1}{\tau^\mathrm{diff}} \bb D^\ast \Delta^\ast \vec c^\ast
    &= \frac{1}{\tau^\mathrm{react}} r_i^\ast(\vec c^\ast),
    &&t \geq 0, \, \vec z \in \Sigma,
    \\
   \frac{1}{\tau^\mathrm{R}} \vec e^k \cdot \partial_{t^\ast} \vec \theta_\mathrm{red} - \vec e^k \cdot \dv_\Sigma(\bb D^{\Sigma,\ast} \nabla_\Sigma^\ast \vec \theta)
    &= \frac{1}{\tau^\mathrm{sorp}} \vec e^k \cdot \vec s^{\Sigma,\ast}(\vec c^\ast, \vec \theta),
    &&k = 1, \ldots, n^\Sigma, \, t \geq 0, \, \vec z \in \Sigma
    \\
   - \frac{1}{\tau^\mathrm{trans}} \bb D^\ast \partial_{\vec n^\ast} \vec c^\ast
    &= \frac{1}{\tau^\mathrm{sorp}} \vec s^{\Sigma,\ast}(\vec c^\ast, \vec \theta),
    &&t \geq 0, \, \vec z \in \Sigma,
    \\
   \vec r^{\Sigma, \ast}(\vec \theta)
    &= \vec 0,
    &&t \geq 0, \, \vec z \in \Sigma.
  \end{align*}
 Inserting the third equation into the second one and multiplying the third line by $\tau^\mathrm{sorp}$, taking the formal limit $\tau^\mathrm{sorp} \rightarrow 0$ then gives
  \begin{align*}
   \frac{1}{\tau^\mathrm{R}} \partial_{t^\ast} \vec c^\ast - \frac{1}{\tau^\mathrm{diff}} \bb D^\ast \Delta^\ast \vec c^\ast
    &= \frac{1}{\tau^\mathrm{react}} \vec r^\ast(\vec c^\ast),
    &&t \geq 0, \, \vec z \in \Omega
    \\
   \frac{1}{\tau^\mathrm{R}} \vec e^k \cdot \partial_{t^\ast} \vec \theta_\mathrm{red} - \frac{1}{\tau^{\mathrm{react},\Sigma}} \vec e^k \cdot \dv_\Sigma^\ast ( \bb D^{\Sigma,\ast} \nabla_\Sigma^\ast \vec \theta )
    &= - \frac{1}{\tau^\mathrm{trans}} \vec e^k \cdot (\bb D \nabla \vec c^\ast \cdot \vec n^\ast),
    &&k = 1, \ldots, n^\Sigma, \, t \geq 0, \, \vec z \in \Sigma
     \\
    \vec r^{\Sigma, \ast}(\vec \theta)
     &= \vec 0,
     &&t \geq 0, \, \vec z \in \Sigma
     \\
    \vec s^{\Sigma, \ast}(\vec c^\ast, \vec \theta)
     &= \vec 0,
     &&t \geq 0, \, \vec z \in \Sigma.
  \end{align*}
 This is the same system as for the fast surface chemistry, ultra-fast sorption limit.
 One therefore expects the same model for equivalently fast sorption and surface chemistry as well; see the next subsection.
 
 \subsubsection{Equivalently fast surface chemistry and sorption}
 For this case, one starts with the full bulk-surface reaction-diffusion-sorption model \eqref{eqn:dimless_bulk}--\eqref{eqn:dimless_transmission}, fixing the ratio $\lambda = \frac{\tau^{\mathrm{react},\Sigma}}{\tau^\mathrm{sorp}} > 0$.
 After multiplying equations \eqref{eqn:dimless_surface} and \eqref{eqn:dimless_transmission} by $\tau^{\mathrm{react},\Sigma} = \lambda \tau^\mathrm{sorp} > 0$ and performing the formal limit $\tau^{\mathrm{react},\Sigma} = \lambda \tau^\mathrm{sorp} \rightarrow 0$, one obtains the two-parameter limit
  \begin{align*}
   \frac{1}{\tau^\mathrm{R}} \partial_{t^\ast} c_i^\ast - \frac{1}{\tau^\mathrm{diff}} d_i^\ast \Delta^\ast c_i^\ast
    &= \frac{1}{\tau^\mathrm{react}} r_i^\ast(\vec c^\ast)
    \\
   0
    &= r_i^{\Sigma,\ast}(\vec \theta) + \lambda s_i^{\Sigma,\ast}(\vec c^\ast, \vec \theta)
     \\
   0
    &= s_i^{\Sigma,\ast}(\vec c^\ast, \vec \theta).
  \end{align*}
 Here, the static relations
  \begin{align*}
   \vec r^{\Sigma, \ast}(\vec \theta) + \lambda \vec s^{\Sigma, \ast}(\vec c^\ast, \vec \theta)
    &= \vec 0,
    \\
   \vec s^{\Sigma, \ast}(\vec c^\ast, \vec \theta)
    &= \vec 0,
  \end{align*}
 are equivalent to $\vec s^{\Sigma,\ast}(\vec c^\ast, \vec \theta) = \vec r^{\Sigma, \ast}(\vec \theta) = \vec 0$, and this, therefore, leads to the same two-parameter limit system as before.
 As a result, concerning the limit model it does not matter whether the sorption or surface chemistry take place equivalently fast, or one of these processes is ultra-fast. The general form one always obtains is
  \begin{align*}
   \partial_t \vec c + \dv \bb J
    &= \vec r(\vec c),
    &&t \geq 0, \, \vec z \in \Omega
    \\
   \vec e^k \cdot \partial_t \vec c^\Sigma + \vec e^k \cdot \dv_\Sigma \bb J^\Sigma
    &= \vec e^k \cdot (\bb J \cdot \vec n),
    &&k = 1, \ldots, n^\Sigma, \, t \geq 0, \, \vec z \in \Sigma
     \\
    \vec r^\Sigma(\vec c^\Sigma)
     &= \vec 0,
     &&t \geq 0, \, \vec z \in \Sigma
     \\
    \vec s^\Sigma(\vec c, \vec c^\Sigma)
     &= \vec 0,
     &&t \geq 0, \, \vec z \in \Sigma.
  \end{align*}
 In this sense, the fast limits for the sorption and the surface chemistry are compatible.
 
 \subsubsection{Equivalent formulation of the sorption and surface chemistry equilibrium condition}
 
 In subsection \ref{subsubsec:fast_sorption} it has been noted that the sorption rates are modelled such that $s_i(\vec c, \vec c^\Sigma) = 0$ if and only if the values of the bulk and surface chemical potentials of species $A_i$ and $A_i^\Sigma$ coincide on the surface, $\mu_i^\Sigma = \mu_i|_\Sigma$. This observation may now be used to remove $\vec c^\Sigma$ from the fast-sorption--fast-surface-chemistry limit model and replace the two equilibrium conditions $\vec r^\Sigma(\vec c^\Sigma) = \vec 0$ and $\vec s^\Sigma(\vec c|_\Sigma, \vec c^\Sigma) = \vec 0$ by a single equilibrium condition $\vec r^\mathrm{b}(\vec c|_\Sigma) = \vec 0$. To this end, note that by assumption \ref{assmpt:detailled_balance} the surface chemistry only has detailed-balance equilibria and, hence, $\vec r^\Sigma(\vec c^\Sigma) = \vec 0$ if and only if $\mathcal{A}_a^\Sigma = \sum_{i=1}^N \mu_i^\Sigma \nu_i^{\Sigma,a} = 0$ for all surface chemical reactions $a = 1, \ldots, m^\Sigma$. Inserting the sorption equilibrium condition $\mu_i^\Sigma = \mu_i|_\Sigma$, this means that
  \[
   \sum_{i=1}^N \mu_i|_\Sigma \nu_i^{\Sigma,a}
    = 0,
    \quad
    a = 1, \ldots, m^\Sigma
  \]
 which is the equilibrium condition for the analogous set of chemical reactions \emph{in the bulk} (but evaluated at the boundary $\Sigma = \partial \Omega$):
  \[
   \sum_{i=1}^N \alpha_i^{\Sigma,a} A_i
    \rightleftharpoons \sum_{i=1}^N \beta_i^{\Sigma,a} A_i,
    \quad
    a = 1, \ldots, m^\Sigma.
  \]
 Denoting the reaction rates belonging to this ensemble of bulk chemical reactions by $\vec r^\mathrm{b}(\vec c)$, the equilibrium condition $\vec s^\Sigma(\vec c|_\Sigma, \vec c^\Sigma) = \vec 0 = \vec r^\Sigma(\vec c^\Sigma)$ is equivalent to the nonlinear quasi-static boundary condition $\vec r^\mathrm{b}(\vec c|_\Sigma) = \vec 0$ on $\Sigma$.
 \subsection{Three-parameter limits}
 \label{subsec:three_parameter_limit}
  One may consider the case where not only the surface chemistry and the sorption process are ultra-fast, but the transmission between bulk and surface is fast as well, i.e.\
    \[
     \tau^\mathrm{sorp}_\mathrm{slow}, \tau^{\mathrm{react},\Sigma}_\mathrm{slow} \ll \tau^\mathrm{trans} \ll \tau^\mathrm{R}, \tau^{\mathrm{diff},\Sigma}.
    \]
  For the motivation of the fast transmission case, cf.\ the one-parameter limit for fast transmission in subsection \ref{subsubsec:fast_transmission}.
  There it has been motivated why especially the relation $\tau^\mathrm{trans} \lesssim \tau^\mathrm{sorp}_\mathrm{slow} \ll \tau^\mathrm{R}, \tau^{\mathrm{diff},\Sigma}$ is very reasonable.
  As for the two-parameter limits, it is not important which of the processes is faster than the others, e.g.\ it does not matter whether the surface chemical reactions are fast or even ultra-fast.
  Therefore, it is enough to establish the model by considering the case $\tau^\mathrm{trans} \rightarrow 0$ in the fast sorption, fast surface reaction model, leading to $\vec e^k \cdot (\bb J \cdot \vec n) = \vec 0$ and the following reduced model
  \begin{align*}
   \frac{1}{\tau^\mathrm{R}} \partial_{t^\ast} \vec c^\ast - \frac{1}{\tau^\mathrm{diff}} \bb D^\ast \Delta^\ast \vec c^\ast
    &= \frac{1}{\tau^\mathrm{react}} \vec r^\ast(\vec c^\ast),
    &&t \geq 0, \, \vec z \in \Omega
    \\
   \frac{1}{\tau^\mathrm{trans}} \vec e^k \cdot \bb D^\ast \partial_{\vec n^\ast} \vec c^\ast
    &= 0,
    &&k = 1, \ldots, n^\Sigma, \, t \geq 0, \, \vec z \in \Sigma
    \\
   \vec r^{\mathrm{b}, \ast}(\vec c^\ast|_\Sigma)
    &= \vec 0,
    &&t \geq 0, \, \vec z \in \Sigma,
  \end{align*}  
 or, returning to the variable $\vec c$:
  \begin{align*}
   \partial_t \vec c - \bb D \Delta \vec c
    &= \vec r(\vec c),
    &&t \geq 0, \, \vec z \in \Omega,
    \\
   - \vec e^k \cdot \bb D \partial_{\vec n} \vec c
    &= 0,
    &&k = 1, \ldots, n^\Sigma, \, t \geq 0, \, \vec z \in \Sigma,
    \\
   \vec r^\mathrm{b}(\vec c|_\Sigma)
    &= \vec 0,
    &&t \geq 0, \, \vec z \in \Sigma,
  \end{align*}  
 so that, for general reaction-diffusion-systems, the limit model reads as
  \begin{align*}
   \partial_t \vec c + \dv \bb J
    &= \vec r(\vec c),
    &&t \geq 0, \, \vec z \in \Omega,
    \\
   \vec e^k \cdot (\bb J \cdot \vec n)
    &= \vec 0,
    &&k = 1, \ldots, n^\Sigma, \, t \geq 0, \, \vec z \in \Sigma,
     \\
    \vec r^\mathrm{b}(\vec c|_\Sigma)
     &= \vec 0,
     &&t \geq 0, \, \vec z \in \Sigma.
  \end{align*}
 As one sees, in this reduced model, the dynamic PDE for the surface concentrations $\vec c^\Sigma$ is fully replaced by quasi-static relations on $\vec c$ and $\bb J \cdot \vec n$, so that a bulk-reaction-diffusion system with nonlinear, mixed-type boundary conditions results.
 \newline
 To get an idea, how for concrete reaction models the resulting PDE system and its boundary conditions look like, consider the following simple three-component model as a prototype example:
 \begin{example}[Three component system]
  Consider a three component system with no bulk chemistry and a surface reaction mechanism of the type
   \[
    A_1^\Sigma + A_2^\Sigma \rightleftharpoons A_3^\Sigma.
   \]
  The reaction rate is modelled by
   \[
    \vec r^\Sigma(\vec \theta)
     = \vec \nu^\Sigma \left( \kappa^f \theta_1 \theta_2 - \kappa^b \theta_3 \theta_0 \right),
     \quad
     \text{where }
     \vec \nu^\Sigma = (-1,-1,1)^\mathsf{T} \text{ and } \kappa^f, \kappa^b > 0.
   \]
  Suitable conservation vectors $\vec e^1$, $\vec e^2$ are given by, e.g.,
   \[
    \vec e^1 = (1,0,1)^\mathsf{T} \quad \text{and} \quad \vec e^2 = (0,1,1)^\mathsf{T}.
   \]
  Moreover, the sorption rate is modelled according to the (single-site) Langmuir model
   \[
    s_i^\Sigma(\vec c, \vec \theta)
     = k_i^\mathrm{ad} c_i \theta_0 - k_i^\mathrm{de} \theta_i,
     \quad
     i = 1, \ldots, 3.
   \]
  For Fickian diffusion in the bulk, the fast sorption, fast surface reaction, fast bulk-surface transport model
  \begin{align*}
   \partial_t \vec c + \dv \bb J
    &= \vec r(\vec c),
    &&t \geq 0, \, \vec z \in \Omega,
    \\
   \vec e^k \cdot (\bb J \cdot \vec n)
    &= \vec 0,
    &&k = 1, 2, \, t \geq 0, \, \vec z \in \Sigma,
     \\
    \vec r^\mathrm{b}(\vec c)
     &= \vec 0,
     &&t \geq 0, \, \vec z \in \Sigma
  \end{align*}
  then takes the particular form
  \begin{align*}
   \partial_t c_i - d_i \Delta c_i
    &= 0,
    &&i = 1, 2, 3, \, t \geq 0, \, \vec z \in \Omega,
    \\
   - (d_k \partial_{\vec n} c_k + d_3 \partial_{\vec n} c_3)
    &= 0,
    &&k = 1, 2, \, t \geq 0, \, \vec z \in \Sigma,
     \\
    k^f \theta_1 \theta_2 - k^b \theta_0 \theta_3
     &= 0,
     &&t \geq 0, \, \vec z \in \Sigma,
     \\
    k_i^\mathrm{ad} c_i \theta_0 - k_i^\mathrm{de} \theta_i
     &= 0,
     &&i = 1,2,3, \, t \geq 0, \, \vec z \in \Sigma,
  \end{align*}
 or, after solving the latter condition for $\vec \theta_\mathrm{red}$ and inserting into the third relation:
  \begin{align*}
   \partial_t c_i - d_i \Delta c_i
    &= 0,
    &&i = 1, 2, 3, \, t \geq 0, \, \vec z \in \Omega,
    \\
   - (d_k \partial_{\vec n} c_k + d_3 \partial_{\vec n} c_3)
    &= 0,
    &&k = 1, 2, \, t \geq 0, \, \vec z \in \Sigma,
     \\
    c_1 c_2 - \kappa c_3
     &= 0,
     &&t \geq 0, \, \vec z \in \Sigma
  \end{align*}
 with $\kappa = \frac{\kappa^b k_1^\mathrm{de} k_2^\mathrm{de} k_3^\mathrm{ad}}{\kappa^f k_1^\mathrm{ad} k_2^\mathrm{ad} k_3^\mathrm{de}} > 0$.
 \end{example}

 \section{Mathematical analysis of a prototype model problem}
 
 This section is devoted to the three-component model problem introduced above, which serves as a first example for those systems which could arise from the fast-sorption--fast-surface-chemistry limit.
 The results on local-in-time well-posedness, positivity of solutions, blow-up criteria, and a-priori bounds will be extended to more general reaction-diffusion-sorption systems in the forthcoming paper \cite{AuBo19+2}.
 
 \begin{model}[Three component model with chemical reactions on the surface]
  Let $N = 3$ and assume that $\vec r \equiv \vec 0$, i.e.\ no reactions occur within the bulk phase, whereas on the surface the following reversible chemical reaction takes place:
   \[
    A_1^\Sigma + A_2^\Sigma
     \rightleftharpoons A_3^\Sigma.
   \]
  In the bulk phase, diffusion is modelled by standard Fickian diffusion, leading to
   \[
    \partial_t c_i - d_i \Delta c_i
     = 0,
     \quad
     i=1, 2, 3, \, t \geq 0, \, \vec z \in \Omega
   \]
  with constant diffusivities $d_i > 0$, $i = 1, 2, 3$.
  On the surface, a fast sorption and fast reaction model with fast transport between bulk and surface is assumed, i.e.\
   \begin{align*}
    \kappa^f \frac{k_1^\mathrm{ad} k_2^\mathrm{ad}}{k_1^\mathrm{de} k_2^\mathrm{de}} c_1 c_2 - \kappa^b \frac{k_3^\mathrm{ad}}{k_3^\mathrm{de}} c_3
     &= 0,
     &&t \geq 0, \, \vec z \in \Sigma,
     \\
    - \vec e^k \cdot (\bb D \partial_{\vec n} \vec c)
     &= 0,
     &&k = 1, 2, \, t \geq 0, \, \vec z \in \Sigma
   \end{align*}
  where $\vec e^1 = (1,0,1)^\mathsf{T}$ and $\vec e^2 = (0,1,1)^\mathsf{T}$ are the conservation vectors under the chemical reaction on the surface.
  This reaction has the stochiometric vector $\vec \nu^{\Sigma,1} = (-1,-1,1)^\mathsf{T} \in \R^3$, and $\bb D = \operatorname{diag} (d_1, d_2, d_3) \in \R^{3 \times 3}$, so that, for $\kappa := \frac{\kappa^b k_1^\mathrm{de} k_2^\mathrm{de} k_3^\mathrm{ad}}{\kappa^f k_1^\mathrm{ad} k_2^\mathrm{ad} k_3^\mathrm{de}} > 0$, the condensed form of the limit model reads as
   \begin{align*}
    \partial_t c_i - d_i \Delta c_i
     &= 0,
     &&i=1, 2, 3, \, t \geq 0, \, \vec z \in \Omega,
     \\
    c_1 c_2 - \kappa c_3
     &= 0,
     &&t \geq 0, \, \vec z \in \Sigma,
     \\   
    - d_1 \partial_{\vec n} c_1 - d_3 \partial_{\vec n} c_3      
     &= 0,
     &&t \geq 0, \, \vec z \in \Sigma,
     \tag{MP}
     \label{MP}
     \\   
    - d_2 \partial_{\vec n} c_2 - d_3 \partial_{\vec n} c_3      
     &= 0,
     &&t \geq 0, \, \vec z \in \Sigma.
   \end{align*}
 \end{model}
  This prototype problem will the investigated here, where the first results are:
   \begin{enumerate}
    \item
     Local-in-time existence and uniqueness of strong solutions to \eqref{MP} in the class $\vec c \in \WW^{(1,2),p}(J \times \Omega; \R^3)$, where $J = [0,T]$ for some $T > 0$, and $p > \frac{n + 2}{2}$, can be ensured, provided the initial data satisfy certain regularity and compatibility conditions. In this case, the solution depends continuously on the initial data $\vec u_0 \in \BB_{pp}^{2-2/p}(\Omega;\R^3)$.
    \item
     As a by-product, the construction of the strong solution provides an abstract blow-up criterion for global in time strong solutions.
    \item
     Moreover, for positive initial data, the solution stays positive and immediately becomes strictly positive for non-negative, but non-zero initial data.
    \item
     Lastly, some a-priori bounds of type $\LL_\infty \LL_1$, $\LL_1 \LL_\infty$, $\LL_2 \LL_2$ and an entropy dissipation inequality hold true.
   \end{enumerate}
  For the notation used here, cf.\ the next subsection on notation.

 \subsection{Some comments on existing literature on analysis of reaction-diffusion systems}
 
 Reaction-diffusion systems constitute a highly relevant and interesting research topic in mathematical modelling, mathematical analysis and numerical simulation of systems relevant for chemical engineering. Unsurprisingly, there is abundant literature on several aspects of reaction-diffusion systems. Within this subsection, a short overview on some related results will be given. Obviously, such an overview can not be exhaustive by any means, and clearly the selection of cited papers may be biased to some extend.
 Here, the selection is partly based on the references given in the excellent survey by M.~Pierre \cite{Pie10}.\newline
 Without doubt, the monograph \cite{LaSoUr68} of O.A.~Ladyzenskaya, V.A.~Solonnikov and N.N.~Ural'ceva has to be included in this list, as it marks maybe the most important cornerstone for the modern theory on parabolic systems, providing  $\LL_p$-maximal regularity and optimality for a very rich class of parabolic evolution equations. More specific on the topic of reaction diffusion equations are the by now classical books by D.~Henry \cite{Hen81} and F.~Rothe \cite{Rot84}. M.~Pierre contributed a vast amount to the literature, e.g.\ \cite{MarPie91}, \cite{Pie03}, \cite{PieSch97}, on the analysis of reaction diffusion systems, with emphasis on global existence or blow-up phenomena. Somehow related to the dimensional analysis considered here are the papers \cite{BotPie10} and \cite{BotPie12} by the second author and M.~Pierre, where fast-reaction-limits for chemical reactions taking place on different time scales have been considered. All these references, however, have one thing in common: They all treat bulk reaction-diffusion-systems with given boundary conditions on the surface of the chemical reactor, e.g.\ homogeneous and non-homogeneous Dirichlet boundary conditions, Neumann boundary conditions or Robin boundary conditions.\newline
 Over the years, the abstract mathematical tools for analysing parabolic systems have considerably improved, e.g.\ with the theory of $\mathcal{R}$-boundedness by E.~Berkson and T.A.~Gillespie \cite{BerGil94}, P.~Cl\'ement et al.\ \cite{{CleDePSukWit00}}, N.J.~Kalton and L.~Weis \cite{KalWei01}.
 For a large class of problems applicable criteria for $\LL_p$-maximal regularity have been established by R.~Denk, M.~Hieber and the late J.~Pr\"uss \cite{DeHiPr03}, \cite{DeHiPr07} and generalised to abstract bulk-surface type systems by R.~Denk, J.~Pr\"uss and R.~Zacher \cite{DePrZa08}. The latter haven been used recently by R.~Schnaubelt to give abstract results on well-posedness and asymptotic behaviour of semilinear bulk-surface reaction-diffusion systems \cite{Sn17}. D.~Bothe, M.~K\"ohne, S.~Maier and J.~Saal \cite{BoKoMaSa17} considered a bulk-surface reaction-diffusion system with inflow and and outflow on a cylindrical domain, proving well-posedness and global existence of strong solutions. \newline
Note that another direction of generalisation of above results are quasilinear instead of semilinear evolution equation. As there is experimental evidence \cite{DuTo62} that the standard Fickian diffusion model is not appropriate to describe situations of non-dilute species in a fluid mixture, and continuum-thermodynamic considerations show that these models are not thermodynamically consistent as they do not obey to the entropy principle \cite{BoDr15}, quite recently alternative models for diffusion of fluids have become more and more popular also in the mathematical community, in particular the \emph{Maxwell-Stefan approach}, see e.g.\ \cite{Bo11}, \cite{HeMePrWi17}, \cite{SoOrMaBo19}.
Furthermore, we refer to the article \cite{BotRol15} of the second author on global existence for reaction-diffusion systems with anisotropic diffusion, i.e.\ a diffusion matrix of the form $D_i = D_i(t, \vec z, \vec c)$.

 \subsection{Some notation and preliminaries}
 Throughout, $\Omega \subseteq \R^n$ denotes a bounded domain with compact boundary $\Sigma = \partial \Omega$ of class $\CC^2$, at least.
 The space $\CC(\overline{\Omega})$ denotes the space of continuous functions and, given $k \in \N_0$ and $\alpha \in (0,1]$, $\CC^k(\overline{\Omega})$ and $\CC^{k+\alpha}(\overline{\Omega})$ are the spaces of $k$-times continuously differentiable functions, respectively of $k$-times continuously differentiable functions with derivatives of order $k$ in the H\"older space $\CC^\alpha(\overline{\Omega})$.
 Moreover, given $p \in [1,\infty]$ and $\Omega \subseteq \R^n$ a (not necessarily bounded) Lebesgue measurable set, the Lebesgue spaces of function classes of Lebesgue measurable functions $f$ such that $\int_\Omega \abs{f}^p \dd \vec z < \infty$ is $\LL_p(\Omega)$, and as usual a function $f$ is identified with its equivalence class $[f] \in \LL_p(\Omega)$ of measurable functions which coincide a.e.\ with $f$. For $k \in \N_0$ or $s \in \R_+$, the symbols $\WW^k_p(\Omega)$ and $\WW^s_p(\Omega)$ denote the Sobolev spaces and Sobolev--Slobodetskii spaces of orders $k$ and $s$, respectively, and Besov spaces $\BB_{pq}^s(\Omega)$ will only be considered for the case $p = q$, noting that $\BB_{pp}^s(\Omega) = \WW^s_p(\Omega)$ for $s \in \R_+ \setminus \N_0$, but $\BB_{pp}^k(\Omega) \neq \WW^k_p(\Omega)$ for $k \in \N_0$.
 All these spaces are equipped with their standard Banach space norms $\norm{\cdot}_\infty$, $\norm{\cdot}_{\CC^k}$, $\norm{\cdot}_{\LL_p}$, $\norm{\cdot}_{\WW^k_p}$ etc., and for sufficiently regular boundary there also exist their surface versions $\CC(\Sigma)$, $\CC^{k + \alpha}(\Sigma)$, $\LL_p(\Sigma)$, $\WW^k_p(\Sigma)$ etc., as well as their Banach space $E$-valued versions, e.g.\ $\LL_p(\Omega; E)$, which are Banach spaces as well.
 As parabolic evolution equations of second order are being considered, also the anisotropic function spaces
  \begin{align*}
   \CC^{(1,2) \cdot (k + \alpha)}(\overline{J} \times \overline{\Omega})
    &= \CC^{k+\alpha, 2(k+\alpha)}(\overline{J} \times \overline{\Omega})
    := \CC^{k+\alpha}(\overline{J}; \CC(\overline{\Omega})) \cap \CC(\overline{J}; \CC^{2(k+\alpha)}(\overline{\Omega})),
    \\
   \WW_p^{(1,2) \cdot s}(J \times \Omega)
    &= \WW_p^{(s,2s)}(J \times \Omega)
    = \WW_p^s(J; \LL_p(\Omega)) \cap \LL_p(J; \WW_p^{2s}(\Omega))
  \end{align*}
 for intervals $J \subseteq \R$, which are Banach spaces for their respective norms $\norm{\cdot}_{\CC^{(1,2) \cdot (k+\alpha)}}$ and $\norm{\cdot}_{\WW_p^{(1,2) \cdot s)}}$, will be used.

 \subsection{Local-in-time existence of strong solutions for the model problem}
 
 Local-in-time existence of strong solutions can be established via the contraction mapping principle.
 
 \begin{theorem}[Local-in-time existence of strong solutions]
  \label{thm:l-it_existence}
  Let $p > \frac{n + 2}{2}$ and assume that $\Omega \subseteq \R^n$ is a bounded domain of class $\partial \Omega \in \CC^2$. Then the model problem \eqref{MP} admits a unique local-in-time strong solution which continuously depends on the initial datum $\vec c^0 \in \BB_{pp}^{2-2/p}(\Omega)$ if and only if
   \[
    \vec c^0
     \in \II_p(\Omega)
      := \{ \vec c^0 \in \BB_{pp}^{2-2/p}(\Omega; \R_+^3): \,
      \kappa c_1^0 c_2^0 - c_3^0 = 0, \,
      d_i \partial_{\vec n} c_i^0 + d_3 \partial_{\vec n} c_3^0 = 0 \, (i = 1,2) \text{ a.e.\  on } \Sigma \, \}.
   \]
  More precisely, for every $\vec c_0^\ast \in \II_p(\Omega)$, there are $T > 0$, $\varepsilon > 0$ and $C > 0$ such that
   \begin{enumerate}
    \item
     For all $\vec c_0 \in \II_p(\Omega)$ with $\norm{\vec c_0 - \vec c_0^\ast}_{\BB_{pp}^{2-2/p}(\Omega)} < \varepsilon$, there is a unique strong solution $\vec c \in \WW_p^{(1,2)}(J \times \Omega; \R^3)$ of \eqref{MP} for $J = [0,T]$,
    \item
     for any two initial values $\vec c_0, \tilde{\vec c}_0 \in \II_p(\Omega)$ with $\norm{\vec c_0 - \vec c_0^\ast}_{\BB_{pp}^{2-2/p}(\Omega)}, \norm{\tilde{\vec c}_0 - \vec c_0^\ast}_{\BB_{pp}^{2-2/p}(\Omega)} < \varepsilon$ and corresponding strong solutions $\vec c, \tilde {\vec c} \in \WW_p^{(1,2)}(J \times \Omega; \R^3)$ one has
      \[
       \norm{\vec c - \tilde {\vec c}}_{\WW_p^{(1,2)}(J \times \Omega)}
        \leq C \norm{\vec c_0 - \tilde {\vec c}_0}_{\BB_{pp}^{2-2/p}(\Omega)},
      \]
    \item
     any strong solution $\vec c \in \WW_p^{(1,2)}(J \times \Omega)$ can be extended in a unique way to a maximal strong solution $\vec c: [0, T_\mathrm{max}) \times \Omega \rightarrow \R_+^3$ with $\vec c \in \WW_p^{(1,2)}((0,T) \times \Omega; \R_+^3)$ for every $T \in (0, T_\mathrm{max})$.
    \item
     If $T_\mathrm{max} < \infty$, then $\norm{\vec c(t,\cdot)}_{\BB_{pp}^{2-2/p}(\Omega)} \rightarrow \infty$ as $t \nearrow T_\mathrm{max}$.
   \end{enumerate}
 \end{theorem}
 
 \begin{proof}
  The following proof relies on a decomposition into three linear subproblems which can be solved successively, and a fixed point argument based on the contraction mapping principle.
  First, observe that problem \eqref{MP} with initial data $\vec c(0,\cdot) = \vec c^0$ is equivalent to the following system of scalar reaction-diffusion systems:
   \begin{equation}
    \begin{cases}
     \partial_t c_1 - d_1 \Delta c_1
      = 0,
      &\text{in } (0,T) \times \Omega,
      \\
     - d_1 \partial_{\vec n} c_1
      = g_1,
      &\text{on } (0,T) \times \Omega,
      \\
     c_1(0,\cdot)
      = c_1^0,
      &\text{in } \Omega,
    \end{cases}
    \tag{\ref{MP}-1}
    \label{MP-1}
   \end{equation}
   \begin{equation}
    \begin{cases}
     \partial_t c_2 - d_2 \Delta c_2
      = 0,
      &\text{in } (0,T) \times \Omega,
      \\
     - d_2 \partial_{\vec n} c_2
      = g_2,
      &\text{on } (0,T) \times \Omega,
      \\
     c_2(0,\cdot)
      = c_2^0,
      &\text{in } \Omega,
    \end{cases}
    \tag{\ref{MP}-2}
    \label{MP-2}
   \end{equation}
   \begin{equation}
    \begin{cases}
     \partial_t c_3 - d_3 \Delta c_3
      = 0,
      &\text{in } (0,T) \times \Omega,
      \\
     c_3
      =h,
      &\text{on } (0,T) \times \Omega,
      \\
     c_3(0,\cdot)
      = c_3^0,
      &\text{in } \Omega
    \end{cases}
    \tag{\ref{MP}-3}
    \label{MP-3}
   \end{equation}
 for $g_1 = g_2 = d_3 \partial_{\vec n} c_3$ and $h = \kappa^{-1} c_1 c_2$.
 Fixing $p \in (\frac{n+2}{2}, \infty)$, initial data $\vec c^0 \in \BB_{pp}^{2-2/p}(\Omega; \R^3)$  such that $c_3^0|_\Sigma = \kappa (c_1^0 c_2^0)|_\Sigma$ (if $p > \frac{3}{2}$) as well as $d_1 \partial_{\vec n} c_1^0 = d_2 \partial_{\vec n} c_2^0 = d_3 \partial_{\vec n} c_3^0$ (if $p > 3$) and positive constants $T_0, \rho_0 > 0$, one may then define for $\rho \in (0, \rho_0]$ and $T \in (0, T_0]$ the map
  \[
   \Phi: \DD_0 \subseteq \WW_p^{(1,2)}((0,T) \times \Omega) \rightarrow \WW_p^{(1,2)}((0,T) \times \Omega)
  \]
 according to
  \[
   \DD_0
    = \{ c \in \WW_p^{(1,2)}((0,T) \times \Omega): \, c(0,\cdot) = c_3^0, \, \norm{c(t,\cdot) - c_0^3}_{\BB_{pp}^{2-2/p}(\Omega)} \leq \rho \quad (\forall t \in (0, T)) \, \}
  \]
 and $\Phi(c_3) = \tilde c_3$ which is constructed in the following way.
  Let $\tilde c_1, \tilde c_2 \in \WW_p^{(1,2)}((0,T) \times \Omega)$ be the (unique) solutions to the linear parabolic problems \eqref{MP-1} and \eqref{MP-2}, respectively, for the inhomogeneity $g_1 = g_2 = d_3 \partial_{\vec n} c_3$.
  Then $\tilde c_3 \in \WW_p^{(1,2)}((0,T) \times \Omega)$ is defined as the unique solution to \eqref{MP-3} for $h = \kappa^{-1} (\tilde c_1 \tilde c_2)|_\Sigma$.
  Thereby, the map $\Phi$ is well-defined, and it remains to adjust $T$ and $\rho$ such that $\Phi$ is a strictly contractive self-mapping on the closed subset $\DD_0$ of $\WW_p^{(1,2)}((0,T) \times \Omega)$.
  To this end, note that the maximal regularity constants $C_\mathrm{D}(T) > 0$ for the parabolic Dirichlet problem \eqref{MP-3} and $C_\mathrm{N}(T) > 0$ for the Neumann problems \eqref{MP-1} and \eqref{MP-2} can be chosen independently of $T \in (0,T_0)$, i.e.\ for the corresponding solutions it holds that
   \begin{align*}
    \norm{c_1}_{\WW_p^{1,2}((0,T) \times \Omega}
     &\leq C_\mathrm{N} \left( \norm{c_1^0}_{\BB_{pp}^{2-2/p}(\Omega)} + \norm{g_1}_{\WW_p^{(1,2) \cdot (1/2 - 1/(2p))}((0,T) \times \Sigma)} \right)
     \\
    \norm{c_3}_{\WW_p^{1,2}((0,T) \times \Omega}
     &\leq C_\mathrm{D} \left( \norm{c_3^0}_{\BB_{pp}^{2-2/p}(\Omega)} + \norm{h}_{\WW_p^{(1,2) \cdot (1 - 1/(2p))}((0,T) \times \Sigma)} \right),
     \quad
     T \in (0, T_0).
   \end{align*}
 Similarly, the embedding constants $C_\mathrm{emb}(T)$ and $C_\mathrm{emb}'(T)$ for the embeddings
  \begin{align*}
   \norm{c|_\Sigma}_{\WW_p^{(1,2) \cdot (1 - 1/(2p))}((0,T) \times \Sigma)}
    &\leq C_\mathrm{emb} \norm{c}_{\WW_p^{(1,2)}((0,T) \times \Omega)},
    \\
   \norm{\partial_{\vec n} c|_\Sigma}_{\WW_p^{(1,2) \cdot (1/2 - 1/(2p))}((0,T) \times \Sigma)}
    &\leq C_\mathrm{emb}' \norm{c}_{\WW_p^{(1,2)}((0,T) \times \Omega)},
    \quad
    T \in (0,T_0),
  \end{align*}
 can be chosen independently of $T \in (0, T_0)$. Finally, the map $(c_1, c_2) \mapsto c_1 c_2$ is, for $p > \frac{n+2}{2}$, continuous as a map from $\WW_p^{(1,2)}((0,T) \times \Omega; \R^2)$ to $\WW_p^{(1,2)}((0,T) \times \Omega)$, and Lipschitz continuous on sets with $\norm{c_i(t,\cdot)}_\infty \leq r$, where the Lipschitz constant tends to zero as $T \rightarrow 0$, for $r \in (0, r_0]$ with some fixed $r_0 > 0$.

 \end{proof}
 
 \begin{lemma}
 \label{lem:product}
  Let $J \subseteq \R$ be a bounded interval and $\Omega \subseteq \R^n$ be a bounded domain with $\CC^2$-boundary.
  Then, for every $k \in \N$, the map
   \[
    M_k: \quad
    \WW^{(1,2)}_p(J \times \Omega)^k
     \rightarrow \WW^{(1,2)}_p(J \times \Omega),
     \quad
    \vec v
     \mapsto \prod_{i=1}^k v_i
   \]
  is continuous, provided $p > \frac{n + 2}{2}$.
  In this case, it is Lipschitz continuous on every bounded (in $\norm{\cdot}_\infty$-norm) subset of $\WW_p^{(1,2)}(J \times \Omega)$. Moreover, considering the closed subspace
   \[
    \DD_0 = \{ v \in \WW_p^{(1,2)}((0,T) \times \Omega): \, v|_{t=0} = 0 \}
   \]  
  of $\WW_p^{(1,2)}(J \times \Omega)$ and the Lipschitz constants $\LL_{\vec \rho, T}$ for $M_k: \prod_{k=1}^k \BB_{\rho_k}^{\DD_0}(0) \rightarrow \DD_0$, one has that $\LL_{\vec \rho,T} \rightarrow 0$, uniformly for all $0 < \rho_i \leq \rho_0$ ($i = 1, \ldots, k$) and $0 < T \leq T_0$, for $\rho_0, T_0 > 0$ fixed, if one lets $T \rightarrow 0$, or $\rho_i \rightarrow 0$ for some $i \in \{1, \ldots, k\}$.
 \end{lemma}
 \begin{proof}
  If $p > \frac{n+2}{2}$, then $\WW^{(1,2)}_p(J \times \Omega) \hookrightarrow \CC(\overline{J} \times \overline{\Omega})$, and, hence, in particular
   \[
    \norm{\prod_{i=1}^n v_i}_{\LL_p(J \times \Omega)}
     \leq \CC(p,J,\Omega) \prod_{i=1}^k \norm{v_i}_{\CC(\overline{J} \times \overline{\Omega})}
     \leq C'(n,p,J,\Omega) \prod_{i=1}^k \norm{v_i}_{\WW^{(1,2)}_p(J \times \Omega)}.
   \]
  Moreover, the gradient of $\prod_{i=1}^k v_i$ is given by
   \[
    \nabla \left( \prod_{i=1}^k v_i \right)
     = \sum_{i=1}^k \left( \prod_{j \neq i} v_j \right) \nabla v_i
   \]
  and for $p > \frac{n+2}{2}$ one has
   \[
    \norm{\big( \prod_{j \neq i} v_j \big) \nabla v_i}_{\LL_p(J \times \Omega)}
     \leq \prod_{j \neq i} \norm{v_j}_{\CC(\overline{J} \times \overline{\Omega})} \norm{\nabla v_i}_{\LL_p(J \times \Omega)}
     \leq C \prod_{i=1}^k \norm{v_i}_{\WW^{(1,2)}_p(J \times \Omega)}.
   \]
  Similarly,
   \[
    \frac{\partial^2}{\partial x_h \partial x_l} \left( \prod_{i=1}^k v_i \right)
     = \sum_{i=1}^k \frac{\partial^2 v_i}{\partial x_h \partial x_l} \big( \prod_{j \neq i} v_j \big)
      + \sum_{i=1}^k \frac{\partial v_i}{\partial x_h} \big( \sum_{j \neq i} \frac{\partial v_j}{\partial x_l} \big) \big( \prod_{m \neq i,j} v_m \big)
   \]
  with
   \begin{align*}
    \norm{\frac{\partial^2 v_i}{\partial x_h \partial x_l}  \prod_{j \neq i} v_j}_{\LL_p(J \times \Omega)}
     &\leq \norm{\frac{\partial^2}{\partial x_h \partial x_l} v_i}_{\LL_p(J \times \Omega)} \prod_{j \neq i}  \norm{v_j}_{\CC(\overline{J} \times \overline{\Omega})}
     \\
     &\leq C \prod_{j=1}^k \norm{v_j}_{\WW_p^{(1,2)}(J \times \Omega)}
     \intertext{for $i = 1, \ldots, k$ and $h, l = 1, \ldots, n$, respectively,}
     \\
    \norm{\frac{\partial v_i}{\partial x_h} \frac{\partial v_j}{\partial x_l} \prod_{m \neq i,j} v_m}_{\LL_p(J \times \Omega)}
    &\leq \norm{\frac{\partial v_i}{\partial x_h}}_{\LL_{2p}(J \times \Omega)} \norm{\frac{\partial v_j}{\partial x_l}}_{\LL_{2p}(J \times \Omega)} \prod_{m \neq i,j} \norm{v_m}_{\CC(\overline{J} \times \overline{\Omega})}
    \\
    &\leq C \prod_{m=1}^k \norm{v_m}_{\WW_p^{(1,2)}(J \times \Omega)},
   \end{align*}
  for $i \neq j \in \{1, \ldots, k\}$ and $h, l \in \{1, \ldots, n\}$,
  since $\WW_p^{(1,2)}(J \times \Omega) \hookrightarrow \LL_p(J;\WW_p^2(\Omega)) \hookrightarrow \LL_p(J; \WW_{2p}^1(\Omega))$, where the latter holds true for all $p \geq \tfrac{n}{2}$.
  In the same fashion, the estimate
   \[
    \norm{\frac{\partial}{\partial t} \prod_{i=1}^k v_i}
     \leq C \prod_{i=1}^k \norm{v_i}_{\WW_p^{(1,2)}(J \times \Omega)}
   \]
  is established.
  Replacing $\vec v \in \WW_p^{(1,2)}(J \times \Omega)$ by $\vec v - \vec w$ for $\vec v, \vec w \in \WW_p^{(1,2)}(J \times \Omega)$ in the estimates above, shows that the map is Lipschitz continuous on every bounded subset of $\WW_p^{(1,2)}(J \times \Omega)$ as well.
  The last assertion follows from the standard lemma below, by which, for $\vec v \in \DD_0$, the embedding constants $C(T)$ for the embeddings $\WW_p^{(1,2)}((0,T) \times \Omega) \hookrightarrow \CC(\overline{J} \times \overline{\Omega})$ can be chosen independently of $T \in (0, T_0)$.
 \end{proof}
 
 \begin{lemma}
  Let $p \in (\frac{n+2}{2}, \infty)$ and $\Omega \subseteq \R^n$ be a bounded $\CC^2$-domain. Fix $T_0 > 0$. For $T \in (0,T_0]$ and
   \[
    \DD_0(T)
     = \{ u \in \WW_p^{(1,2)}((0,T) \times \Omega): \, u(0) = 0 \},
   \]
  the embedding constant for the continuous embedding
   \[
    \DD_0(T) \hookrightarrow \CC(\overline{J} \times \overline{\Omega})
   \]
  can be chosen independently of $T \in (0, T_0)$, e.g.\ $C = 2^{1/p} \CC(T_0)$ where $\CC(T_0)$ is an embedding constant for $T = T_0$.
 \end{lemma}
 \begin{proof}
  Since for $u \in \DD_0(T)$ one has $u(0) = 0$, it follows that
   \[
    \tilde u(t,\cdot)
     := \begin{cases}
      u(t,\cdot),
      &t \in [0, T],
      \\
     u(T-t,\cdot),
      &t \in (T,2T],
      \\
     0,
      &t > 2T
     \end{cases}
   \]
  defines a function $\tilde u \in \WW_p^{(1,2)}(\R_+ \times \Omega)$ and for its restriction to $[0,T_0] \times \Omega$ it holds that
   \begin{align*}
    \norm{\tilde u}_{\WW_p^{(1,2)}((0,T_0) \times \Omega)}
     &\leq 2^{1/p} \norm{u}_{\WW_p^{(1,2)}((0,T) \times \Omega)},
     \\
    \norm{\tilde u}_\infty
     &= \norm{u}_\infty.
   \end{align*}
  From here the assertion follows easily.
 \end{proof}

 \subsection{Classical solutions}

 For Fickean type reaction-diffusion systems with homogeneous Dirichlet or Neumann boundary data, it is well known that these systems have the property of \emph{instantaneous smoothing}, e.g.\ for the homogeneous heat equation with no-flux boundary conditions
  \[
   \partial_t \vartheta - \Delta \vartheta = 0, \quad \text{in } (0,T) \times \Omega,
    \quad \partial_{\vec n} \vartheta = 0 \quad \text{on } (0,T) \times \Sigma,
    \quad \text{and} \quad
    \vartheta(0,\cdot) = \vartheta^0 \in \BB_{pp}^{2-2/p}(\Omega) 
  \]
 for some $p > 1$ and $\Omega$ a bounded domain with $\CC^\infty$-boundary, one has $\vartheta \in \CC((0,T); \CC^\infty(\overline{\Omega}))$, or, if $\partial \Omega$ is merely of class $\partial \Omega \in \CC^{k + \alpha}$ for some $k \in \N_0$ and $\alpha \in (0,1)$, $\vartheta \in \CC((0,T); \CC^{k+\alpha}(\overline{\Omega}))$. This essentially follows from a bootstrap-type argument, maximal regularity in the spaces $\LL_p(\Omega)$ ($p \in (1, \infty)$) and $\CC^\alpha(\overline{\Omega})$ ($\alpha \in (0,1)$), and the Sobolev-Morrey-embeddings for regular bounded domains.
 In the case of inhomogeneous boundary data such a smoothing cannot be observed, as by the optimality results on parabolic evolution equations with inhomogeneous boundary data, the boundary data a priori need to satisfy the required regularity conditions in the optimality results. Therefore, one can only expect the solution $\vec c$ to lie in the class $\vec c \in \WW_p^{(1,2)}((0,T) \times \Omega;\R^3)$ if $\vec c^0 \in \BB_{pp}^{2-2/p}(\Omega;\R^3)$. On the other hand, for $\Omega$ being a bounded $\CC^{2+\alpha}$-domain and initial data in the class $\CC^{2+\alpha}(\overline{\Omega};\R^3)$ (and subject to the compatibility conditions at the boundary), one would expect the solution $\vec c$ to lie in the class $\vec c \in \CC^{(1,2) \cdot (1 + \alpha/2)}([0,T] \times \overline{\Omega};\R^3) = \CC^{1+\alpha/2}([0,T]; \CC(\overline{\Omega};\R^3)) \cap \CC([0,T]; \CC^{2+\alpha}(\overline{\Omega};\R^3))$.
 \begin{proposition}[Local-in-time existence of classical solutions]
  Let $\Omega \subseteq \R^n$ be a bounded domain of class $\partial \Omega \in \CC^{2+\alpha}$ for some $\alpha \in (0,1)$. If $\vec c^0 \in \CC^{2+\alpha}(\overline{\Omega}; \R^2)$ satisfies the compatibility conditions $c_3^0|_\Sigma = \kappa (c_1^0 c_2^0)|_\Sigma$ and $- d_1 \partial_{\vec n} c_1^0 = - d_2 \partial_{\vec n} c_2^0 = d_3 \partial_{\vec n} c_3^0$ on $\Sigma$, and the boundary values of the initial data are sufficiently small in $\LL^\infty(\Sigma;\R^3)$-norm, then there is $T > 0$ and a unique solution $\vec c \in \CC^{1+\frac{\alpha}{2}, 2 + \alpha}([0,T] \times \overline{\Omega}; \R^3)$ of \eqref{MP}, and it coincides with the $\WW_p^{(1,2)}$-solutions for $p \in (1, \infty)$.
 \end{proposition}
We will not give a full prove here, but only comment on approaches towards improving local-in-time existence of classical solutions.
As already noted, such existence results cannot simply be derived from a smoothing effect (which we do not have globally here, but only smoothing in the interior $\Omega$, i.e.\ smoothness of the solutions in $(0,T) \times \Omega$).
Instead, one may repeat the existence proof as in the $\LL^p$-setting, i.e.\ consider initial data $\vec c^0 \in \CC^{2+\alpha}(\overline{\Omega})$ such that the compatibility conditions $- d_1 \partial_{\vec n} c_1^0 = - d_2 \partial_{\vec n} c_2^0 = d_3 \partial_{\vec n} c_3^0$ as well as $\kappa c_3^0 = c_1^0 c_2^0$ are satisfied on $\Sigma$.
One may then try to decompose the system into three scalar subproblems \eqref{MP-1}--\eqref{MP-3}, analogously to the existence proof in the $\LL^p$-setting.
The map $\Phi$ for which one seeks a fixed point corresponding to a classical solution $\vec c \in \CC^{1+\alpha/2, 2+ \alpha}([0,T] \times \overline{\Omega}; \R^3)$, may then be defined on the complete metric space
 \[
  \mathcal{D}_0
   = \{ c \in \CC^{1+\alpha/2, 2+\alpha}([0,T] \times \overline{\Omega}): \, c(0,\cdot) = 0\}.
 \]
Using H\"older-optimal estimates for the inhomogeneous Dirichlet and Neumann problems for the diffusion equation, and employing interpolation inequalities for functions $f \in \CC^{1+\alpha/2, 2+\alpha}([0,T] \times \overline{\Omega})$ with $f(0,\cdot) = 0$, one can then show that $\Phi$ is a strict contraction $\Phi: \mathcal{D}_0 \rightarrow \mathcal{D}_\rho$, provided
 \[
  C \rho + C'(T) + C \norm{\vec c^0}_{\CC^0},
   < 1
 \]
where $C > 0$ and $C(T) \rightarrow 0$ as $T \rightarrow 0+$.
Clearly, for small initial data $\norm{\vec c^0}_{\CC^0}$ such a choice of $T, \rho > 0$ is possible, which gives local-in-time existence of classical solutions for sufficiently regular data, \emph{provided the initial data are small in supremums norm}.
Clearly, the latter is very unsatisfactory, and we are confident that this smallness condition is unnecessary, i.e.\ can be dropped.
The problem rather lies in the technique used for the proof, more precisely the decomposition into three (quite easy to handle) \emph{scalar} subproblems \eqref{MP-1}--\eqref{MP-3}.
Comparing our proof with the techniques used in \cite[section 8.5.4]{Lun95} for quite general scalar nonlinear parabolic problems with nonlinear boundary conditions, one finds that the approach used there should help to get rid of the smallness condition on $\norm{\vec c^0}_{\CC^0}$.
First note, however, that the results from there cannot directly be applied to the situation considered here:
On the one hand, in \cite[section 8.5.4]{Lun95} scalar equations are being considered, but even more importantly, the results are formulated for \emph{oblique boundary conditions} under the \emph{non-tangentiality condition}.
The prototype fast-surface-limit model \eqref{MP}, however, has boundary conditions of mixed type (flux-coupling conditions $-d_1 \partial_{\vec n} c_1 = -d_2 \partial_{\vec n} c_2 = d_3 \partial_{\vec n} c_3$, i.e.\ first order type boundary conditions, plus a nonlinear Dirichlet coupling-condition $\kappa c_3 = c_1 c_2$, i.e.\ a zero order boundary condition) and using a linearisation analogous to the technique used in \cite[section 8.5.4]{Lun95} hints that one should consider a linearised version of the fast-surface limit model \eqref{MP} of the form
 \begin{align*}
  \partial_t \vec c - \bb D \Delta c
   &= \vec f
   &&\text{in } (0, \infty) \times \Omega,
   \\
  - d_1 \partial_{\vec n} c_1 - d_3 \partial_{\vec n} c_3
   &= g_1
   &&\text{on } (0, \infty) \times \Sigma,
   \\
  - d_2 \partial_{\vec n} c_2 - d_3 \partial_{\vec n} c_3
   &= g_2
   &&\text{on } (0, \infty) \times \Sigma,
   \\
  \kappa c_1 - c_1^0 c_2 - c_2^0 c_3
   &= h
   &&\text{on } (0, \infty) \times \Sigma
   \\
  \vec c(0,\cdot)
   &= \vec c^0
   &&\text{on } \overline{\Omega}
 \end{align*}
and one has to show that this system has optimal H\"older regularity, i.e.\ for every $\vec f \in \CC^{\alpha/2,\alpha}([0,T] \times \overline{\Omega};\R^3)$, $g_1, g_2 \in \CC^{(1+\alpha)/2,1+\alpha}([0,T] \times \Sigma)$, $h \in \CC^{1+\alpha/2,2+\alpha}([0,T] \times \Sigma)$ and $\vec c^0 \in \CC^{2+\alpha}(\overline{\Omega})$, there is a unique solution $\vec c \in \CC^{1+\alpha/2,2+\alpha}([0,T] \times \overline{\Omega}; \R^3)$ which depends continuously on the data, in the sense that
 \[
  \norm{\vec c}_{\CC^{1+\alpha/2,2+\alpha}}
   \leq C \big( \norm{\vec f}_{\CC^{\alpha/2,\alpha}} + \norm{\vec g}_{\CC^{(1+\alpha)/2,1+\alpha}} + \norm{h}_{\CC^{1+\alpha/2,2+\alpha}} + \norm{\vec c^0}_{\CC^{2+\alpha}} \big)
 \]
with some $C > 0$.
To our knowledge, for general systems of mixed type of boundary conditions such optimality results have not been stated explicitly in the literature, but based on the fact that in the $\LL^p$-setting such an extension of $\LL^p$-maximal regularity results is possible, when modifying the proofs for abstract $\LL^p$-maximal regularity results, we strongly believe that such Schauer-type estimates can be derived as well.
Then, the fixed point procedure could be applied in the space
 \[
  \mathcal{D}_0
   = \{\vec c \in \CC^{1+\alpha/2,2+\alpha}([0,T] \times \overline{\Omega}): \, \vec c(0,\cdot) = \vec 0 \}
 \]
and $\CC^{\alpha/2,\alpha}$-maximal regularity in connection with interpolation inequalities for functions $f$ in the class $\CC^{1+\alpha/2,2+\alpha}([0,T] \times \overline{\Omega})$ would provide a unique local-in-time classical solution $\vec c \in \CC^{1+\alpha/2,2+\alpha}([0,T] \times \overline{\Omega}; \R^3)$ without any smallness condition on the initial data $\vec c^0 \in \CC^{2+\alpha}(\overline{\Omega}; \R^3)$, complying to the compatibility conditions $-d_1 \partial_{\vec n} c_1^0 = - d_2 \partial_{\vec n} c_2^0 = d_3 \partial_{\vec n} c_3^0$ and $\kappa c_3^0 = c_1^0 c_2^0$ on $\Sigma$ of the latter.

 \subsection{Positive invariance for the model problem}
 
 In the previous sections, we have shown that for any initial data $\vec c^0 \in \BB_{pp}^{2-2/p}(\Omega;\R^3)$ satisfying the compatibility conditions and for $p > \frac{n+2}{2}$, there is a unique solution $\vec c \in \WW_p^{(1,2)}((0,T) \times \Omega; \R^3)$.
 This can be extended to a non-continuable maximal solution $\vec c: (0,T_\mathrm{max}) \times \Omega \rightarrow \R^3$, where either $T_\mathrm{max} = \infty$, or $\norm{\vec c(t,\cdot)}_{\BB_{pp}^{2-2/p}(\Omega;\R^3)}$ blows up in finite time as $t \nearrow T_\mathrm{max}$.
 \newline
 To have a physically relevant model, however, one aims for positivity of solutions: As $\vec c$ models a vector of molar concentrations, only $c_i(t,\vec z) \geq 0$ for all $t \geq 0$ and $\vec z \in \overline{\Omega}$ is physically significant, and, in particular, for non-negative initial data, the solution should stay non-negative for all times $t \in (0, T_\mathrm{max})$.

 \begin{proposition}[Positive invariance]
  Let $p > \frac{n}{2}$, $\Omega \subseteq \R^n$ be a bounded domain of class $\CC^{2+\alpha}$, let $\vec 0 \leq \vec c_0 \in \CC^{2+\alpha}(\overline{\Omega})$ satisfy the compatibility conditions for \eqref{MP} and $\vec c \in \CC^{1+\alpha/2,2+\alpha}([0,T] \times \overline{\Omega};\R^3)$ be a classical solution to \eqref{MP}.
  Then either $\vec c^0 \equiv \vec 0$ (and hence $\vec c \equiv \vec 0$), or $\vec c(t,\vec z) \in (0, \infty)^3$ for all $t \in (0, T]$ and $\vec z \in \overline{\Omega}$.
 \end{proposition}
 \begin{proof}
  By the weak parabolic maximum principle, for any $t \in [0, T]$, it holds that
   \[
    m(t)
     := \min_{(s,\vec z) \in [0,t] \times \overline{\Omega}} c_i(s,\vec z)
     = \min \{ \min_{\vec z \in \overline{\Omega}} c_i^0(\vec z), \, \min_{(s,\vec z) \in [0,t] \times \Sigma} c_i(s,\vec z)\},
     \quad
     i = 1, 2, 3,
   \]
  and by the strict maximum principle, for every $(t,z) \in (0,T] \times \Omega$ either $c_i(t,\vec z) > m(t)$ or $c_i = \mathrm{const}$ on $[0,t] \times \overline{\Omega}$.
  Let $t_0 := \sup \{t \in [0,T]: c_i \geq 0 \text{ on } [0,t] \times \Omega\}$, and \emph{assume} that there are $i_0 \in \{1, 2, 3\}$ and $z_0 \in \overline{\Omega}$ such that $c_{i_0}(t_0,z_0) = 0$ (otherwise, $t_0 = T$ and $c_i > 0$ on $[0,T] \times \overline{\Omega}$ for $i = 1, 2, 3$).
  \newline
  \emph{First case: $t_0 = 0$.}
  In this case the solution would immediately have negative values for small $t > 0$. 
  Using a small, but strictly positive perturbation of the initial data, one can use the contraction mapping principle to show that for small perturbations obeying the compatibility conditions, classical solutions $\vec c^\varepsilon \in \CC^{1+\alpha/2,2+\alpha}([0,T_\varepsilon] \times \overline{\Omega};\R^3)$ exist and using the arguments below (second and third case), it can be shown that $c_i^\varepsilon > 0$. Letting the perturbation become small, one then obtains that $c_i(t,\vec z) = \lim_{n \rightarrow \infty} c_i^{\varepsilon_n}(t, \vec z) \geq 0$ for small times $t \geq 0$ and $\vec z \in \overline{\Omega}$, so that the case $t_0 = 0$ can be excluded.
  \newline
  \emph{Second case: $t_0 > 0$ and $z_0 \in \Omega$.}
  Then, by the strict maximum principle, it follows that $c_{i_0} \equiv 0$ on $[0,t] \times \overline{\Omega}$, and in particular $c_{i_0}|_{[0,t] \times \Sigma} \equiv 0$ and $\partial_{\vec n} c_{i_0}|_{[0,t] \times \Sigma} = 0$.
  From the boundary conditions $\kappa c_3|_\Sigma = (c_1 c_2)|_\Sigma$ and $\partial_{\vec n} c_1|_\Sigma = \partial_{\vec n} c_2|_\Sigma = - \partial_{\vec n} c_3|_\Sigma$ on $[0,t] \times \Sigma$, it follows that all of these boundary terms are identically zero on $[0,t] \times \Omega$.
  By Hopf's lemma (as $\partial_{\vec n} c_i = 0$) this can only be the case if $c_i(t_0,\cdot)$ is not strictly positive everywhere on $\Omega$ for all $i = 1, 2, 3$, but then by the strict maximum principle already $c_i \equiv 0$ on $[0,t] \times \overline{\Omega}$ for $i = 1, 2, 3$. Accordingly, $\vec c^0 \equiv 0$ and $\vec c \equiv 0$ as solutions are unique.
  \newline
  \emph{Third case: $t_0 > 0$, $c_i > 0$ in $[0,t_0] \times \Omega$ and $z_0 \in \Sigma$:} In this case, by the boundary conditions $c_3(t_0,z_0) = \kappa c_1(t_0,z_0) c_2(t_0,z_0)$, there is $i_1 \neq i_0$ and $3 \in \{i_0, i_1\}$ such that $c_{i_1}(t_0,z_0) = 0$ as well, and, moreover, $d_{i_0} \partial_{\vec n} c_{i_0}(t_0,z_0) = - d_{i_1} \partial_{\vec n} c_{i_1}(t_0,z_0)$.
  By Hopf's lemma, however, both $\partial_{\vec n} c_{i_0}(t_0,z_0), \partial_{\vec n} c_{i_1}(t_0,z_0) < 0$, a contradiction!
  \newline
  These considerations show that either $\vec c$ is identically zero, or component-wise strictly positive, for all $t \in (0, T]$.
 \end{proof}

  \subsection{A-priori bounds on the strong solution of the fast sorption--surface-chemistry model}

 By Theorem \ref{thm:l-it_existence}, a bound on the phase space norm $\norm{\cdot}_{\BB_{pp}^{2-2/p}}$ is enough for establishing global existence of a strong solution. To derive such a bound, is a delicate matter and it is actually unclear how to achieve this.
 On the other hand, for some weaker norms a-priori bounds are valid \emph{for free}.
 The derivation of these is based on the parabolic maximum principle and entropy considerations, highlighting the fruitful interplay between mathematics and physics.
 
 \begin{theorem}[A-priori bounds]
  Let $\vec c^0 \in \II_p(\Omega) \cap \CC^2(\overline{\Omega}; \R^N)$ and $\vec c \in {\CC^{(1,2)}([0, T_\mathrm{max}) \times \overline{\Omega};\R_+^N)}$ be a maximal classical solution to the fast-surface-chemistry--fast-sorption--fast-transmission limit problem
   \begin{align*}
    \partial_t \vec c - \bb D \Delta \vec c
     &= \vec r(\vec c),
     &&t \geq 0, \, \vec z \in \Omega
     \\
    - \vec e^k \cdot \bb D \partial_{\vec n} \vec c|_\Sigma
     &=  \vec 0,
     &&k = 1, \ldots, n^\Sigma, \, t \geq 0, \vec z \in \Sigma
     \\
    \vec r^\mathrm{b}(\vec c|_\Sigma)
     &= \vec 0,
     &&t \geq 0, \vec z \in \Sigma
     \\
    \vec c(0,\cdot)
     &= \vec c^0,
     &&\vec z \in \overline{\Omega}.    
   \end{align*}
  Further assume that there is a conservation vector with strictly positive entries
   \[
    \vec e \in (0, \infty)^N \cap \{\vec \nu^a: a = 1, \ldots, m\}^\bot \cap \{\vec \nu^{\Sigma,a}: a = 1, \ldots, m^\Sigma\}^\bot.
   \]
  Then, for every $T_0 \in (0, T_\mathrm{max}] \cap \R$ there is $C = \CC(T_0) > 0$, also depending on the initial data $\vec c^0$, such that the following a-priori bounds hold true:
   \begin{enumerate}
    \item
     $\LL_\infty \LL_1$--a-priori estimate:
      \[
       \sup_{t \in [0, T_0)} \norm{\vec c(t,\cdot)}_{\LL_1(\Omega;\R^N)}
        \leq C \norm{\vec c^0}_{\LL_1(\Omega; \R^N)};
      \]
    \item
     $\LL_1 \LL_\infty$--a-priori estimate:
      \[
       \sup_{\vec z \in \overline{\Omega}} \norm{c(\cdot,\vec z)}_{\LL_1([0,T_0);\R^N)}
        \leq C;
      \]
    \item
     $\LL_2 \LL_2$--a-priori estimate:
      \[
       \norm{\vec c}_{\LL_2([0,T_0) \times \Omega; \R^N)}
        \leq C;
      \]
    \item
     Moreover, the following \emph{entropy identity} holds true:
      \begin{align*}
       &\int_\Omega c_i(t,\vec z) (\mu_i^0 + \ln \vec c_i(t,\vec z) - 1) \dd \vec z
        \\ &\quad
        + \int_0^t \int_\Omega \sum_{i=1}^N d_i \frac{\abs{\nabla c_i(s,\vec z)}^2}{c_i(s,\vec z)} \dd \vec z \dd s
        + \sum_{a=1}^m \int_0^t \int_\Omega \big( \sum_{i=1}^N \ln(c_i) \nu_i^a \big) \left(\exp \big( \sum_{i=1}^N \ln(c_i) \nu_i^a \big) - 1 \right) \dd \vec z \dd s
        \\ &\quad
       = \int_\Omega c_i^0(\vec z) (\mu_i^0 \ln c_i^0(\vec z) - 1) \dd \vec z,
        \quad
        t \in [0,T_0).
      \end{align*}
   \end{enumerate}
 \end{theorem}
 
 \begin{proof}
  The a-priori bounds can be established analogously to those for standard reaction-diffusion systems, e.g.\ Fickean diffusion in the bulk with no-flux boundary conditions, cf.\ \cite{Pie10}. As this can be done for general reaction-diffusion-sorption systems, details will be presented in \cite{AuBo19+2} and here only a sketch of the proofs will be given. \newline
  For the $\LL_\infty \LL_1$-estimate, one simply considers the time derivative of $\int_\Omega \vec c(t,\vec z) \cdot \vec e \dd \vec z$, and via integration by parts and the no-flux boundary conditions $\vec e^k \cdot \partial_{\vec n} \vec c = 0$ on the conserved quantities, one finds that $\int_\Omega \vec c(t,\vec z) \cdot \vec e \dd \vec z = \int_\Omega \vec c^0(\vec z) \cdot \vec e \dd \vec z$ is independent of $t \geq 0$, hence the results follows from positivity of $\vec c$ and $\vec e \in (0, \infty)^N$.
  \newline
  To get an $\LL_1 \LL_\infty$-a priori estimate, one can employ the conservation vector $\vec e \in (0, \infty)^N$ and consider the function $w$ on $[0,T] \times \overline{\Omega}$ defined by $w(t,\vec z) = \int_0^t \bb D \vec c(s,\vec z) \cdot \vec e \dd \vec z$ which satisfies the differential inequality $\partial_t w - d_\mathrm{max} \Delta w \leq d_\mathrm{max} \vec c^0 \cdot \vec e$ in the bulk and homogeneous no-flux boundary conditions $- d_\mathrm{max} \partial_{\vec n} w = 0$. The a priori estimate can then derived by applying the parabolic maximum principle to $w$.
  \newline
  The $\LL_2 \LL_2$-estimate can be established by considering for $T \in (0, T_\mathrm{max})$, and $d_\mathrm{min} := \min_i d_i$, $d_\mathrm{max} = \max_i d_i > 0$ the integral
   \[
    \int_0^T \int_\Omega (\bb D \vec c(t,\vec z) \cdot \vec e) (\vec c(t,\vec z) \cdot \vec e) \dd \vec z \dd \vec t
     \leq \frac{d_\mathrm{max}}{d_\mathrm{min}} T \norm{\vec c^0 \cdot \vec e}_{L_\infty(\Omega)} \norm{\bb D \vec c^0 \cdot \vec e}_{L_1(\Omega)},
   \]
  and using positivity of $\vec c$.
  \newline
  The entropy identity can be shown as follows.
  By the theorem on derivatives of parameter-integrals, and as the derivative of the function $(0,\infty) \in x \mapsto x (\ln x - 1)$ is $\ln x $ for all $x \in (0, \infty)$, one finds that
   \begin{align*}
    &\frac{\dd}{\dd t} \int_\Omega \sum_{i=1}^3 c_i (\mu_i^0 + \ln (c_i) - 1) \dd \vec z
     \\
     &= \int_\Omega \sum_{i=1}^3 \partial_t c_i (\mu_i^0 + \ln (c_i)) \dd \vec z
     = \int_\Omega \sum_{i=1}^3 d_i \Delta c_i (\mu_i^0 + \ln (c_i))
     \\
     &= - \int_\Omega \sum_{i=1}^3 \frac{d_i \abs{\nabla c_i}^2}{c_i} \dd \vec z
      + \int_\Sigma \sum_{i=1}^3 \mu_i d_i \partial_{\vec n} c_i \dd \sigma(\vec z)
   \end{align*}
  The assertion will be established if $\sum_{i=1}^3 (\mu_i^0 + \ln (c_i)) d_i \partial_{\vec n} c_i = 0$ can be proved.
  From the boundary conditions $\vec e^k \cdot \partial_{\vec n} (\bb D \vec c) = 0$ ($k  = 1,2$), there is a scalar function $\eta: [0,T_0) \times \Sigma \rightarrow \R$ such that $\partial_{\vec n} (\bb D \vec c)|_\Sigma = \eta \, (-1,-1,1)^\mathsf{T} = \eta \vec \nu^{\Sigma,1}$.
  Hence,
   \begin{align*}
    \sum_{i=1}^3 \mu_i d_i \partial_{\vec n} c_i
     = \vec \mu \cdot \partial_{\vec n} (\bb D \vec c)
     = \eta \vec \mu \cdot \vec \nu^{\Sigma,1}
     = \eta \sum_{i=1}^N \mu_i \nu_i^{\Sigma,1}
     = \eta \mathcal{A}^\Sigma
     = 0,
   \end{align*}
  since $\mu_i|_\Sigma = \mu_i^\Sigma$ at all times $t \geq 0$ and positions $\vec z \in \Sigma$ (sorption processes in equilibrium) and $\mathcal{A}^\Sigma_a = 0$ at all times $t \geq 0$ and all positions $\vec z \in \Sigma$ (chemical reaction on the surface in equilibrium).
 Therefore, this contribution to the sum vanishes, and the entropy identity follows by the fundamental theorem of calculus.
 \end{proof}
 
  \subsection{Equilibria}
 In this subsection we compute the (spatial homogeneous) equilibria compatible with the initial condition $\vec c(0,\cdot) = \vec c^0$.
 As the vectors $(1,1,0)^\mathsf{T}$ and $(0,1,1)^\mathsf{T}$ are conservation vectors, i.e.\
  \[
   \int_\Omega (c_1(t,\vec z) + c_3(t,\vec z)) \dd \vec z
    = \int_\Omega (c_1^0 + c_3^0) \dd \vec z,
    \quad
   \int_\Omega (c_2(t,\vec z) + c_3(t,\vec z)) \dd \vec z
    = \int_\Omega (c_2^0 + c_3^0) \dd \vec z
  \]
 are conserved quantities under the evolution of the reaction-diffusion system, we may set
  \[
   a
    = \frac{1}{\abs{\Omega}} \int_\Omega (c_1^0 + c_3^0) \dd \vec z,
    \quad \text{and} \quad
   b
    = \frac{1}{\abs{\Omega}} \int_\Omega (c_2^0 + c_3^0) \dd \vec z
  \]
 Restricting to physically relevant systems, we assume that $\vec c^0 \geq \vec 0$ is pointwise and componentwise non-negative, so that the solution $\vec c \geq \vec 0$ is non-negative as well. If $a = 0$, then $c_1^0 = c_3^0 \equiv 0$ are identically zero, and the solution to the reaction-diffusion-sorption system is clearly given by $\vec c = (0, c_2, 0)^\mathsf{T}$ where $c_2$ is the (unique) solution to the no-flux (homogeneous Neumann) reaction diffusion system
  \[
   \begin{cases}
    \partial_t c_2 - d_2 \Delta c_2
     = 0
     &\text{in } (0, \infty) \times \Omega,
     \\
    - d_2 \partial_{\vec n} c_2|_\Sigma
     = 0,
     &\text{on } (0, \infty) \times \Sigma,
     \\
    c_2(0,\cdot)
     = c_2^0
     &\text{on } \overline{\Omega},
   \end{cases}
  \]
 which (for sufficiently regular initial data) exists and converges to $c_2^\infty = b = \frac{1}{\abs{\Omega}} \int_\Omega c_2^0 \dd \vec z \geq 0$, so that the corresponding equilibrium is $\vec c^\infty = (0, c_2^\infty, 0)^\mathsf{T} \in \R_+^3$.
 Similarly, the case $b = 0$, hence $c_2^0 = c_3^0 \equiv 0$ can be handled, for which the corresponding equilibrium reads $\vec c^\infty = (c_1^\infty, 0, 0)^\mathsf{T} \in \R_+^3$ with $c_1^\infty = a = \frac{1}{\abs{\Omega}} \int_\Omega c_1^0 \dd \vec z$.
 \newline
 Thus, in the following we may and will assume that $a, b > 0$.
 First note, that in this case, the equilibria are spatial homogeneous and strictly positive, as can be seen as follows:
  Since for non-negative initial data, which are not identically zero in one component, this component of the solution becomes strictly positive in $\Omega$ immediately, one may actually deduce that for $a, b > 0$, the solution becomes immediately strictly positive in $\Omega$ \emph{in all components}. E.g.\ assume that $c_3 \equiv 0$ on some small time interval $[0, \delta]$; then from the boundary coupling $- d_1 \partial_{\vec n} c_1|_\Sigma = - d_2 \partial_{\vec n} c_2|_\Sigma = 0$ on $[0, \delta] \times \Sigma$ and $c_1 c_2|_\Sigma = 0$ on $[0, \delta] \times \Sigma$ as well. However, since $\partial_t c_i - d_i \Delta c_i = 0$ and $c_i > 0$ in $[0,\delta] \times \Omega$ for $i = 1, 2$, this contradicts Hopf's boundary value lemma. Similarly, one may exclude the case $c_i \equiv 0$ for some $i \in \{1, 2\}$ on $[0, \delta] \times \overline{\Omega}$. As a result, $c_i > 0$ for all $i = 1, 2, 3$ and $(t,\vec z) \in (0, \infty) \times \overline{\Omega}$.
  In particular, for positive times $t > 0$, the free energy
   \[
    \psi(t)
     = \sum_{i=1}^3 \int_\Omega c_i (\mu_i^0 + \ln(c_i) - 1)
     < \infty
   \]
  and thus, from the free energy dissipation relation it follows that any equilibrium state $\vec c^\infty$ necessarily has to satisfy
   \[
    \sum_{i=1}^3 \int_\Omega \frac{d_i \abs{\nabla c_i^\infty}^2}{(c_i^\infty)^2} \dd \vec z
     = 0.
   \]
  Therefore, any equilibrium state $\vec c^\infty \in \R_+^3$ is constant, and when $a, b > 0$, then $\vec c^\infty \in (0, \infty)^3$.
  Moreover, from the conservation laws one finds that
   \begin{align*}
    a
     = \frac{1}{\abs{\Omega}} \int_\Omega (c_1^0 + c_3^0) \dd \vec z
     \overset{!}{=} c_1^\infty + c_3^\infty,
     \quad \text{and} \quad
    b
     = \frac{1}{\abs{\Omega}} \int_\Omega (c_2^0 + c_3^0) \dd \vec z
     \overset{!}{=} c_2^\infty + c_3^\infty.
   \end{align*}
 To find the equilibrium states compatible with the initial data $\vec c^0$, one thus has to solve the nonlinear problem resulting from the boundary conditions
  \[
   \begin{cases}
    c_3^\infty - \kappa c_1^\infty c_2^\infty
     = 0,
     \\
    \text{under the constraint }
    a = c_1^\infty + c_3^\infty, \, b = c_2^\infty + c_3^\infty.
   \end{cases}
  \]
 which, by inserting the relation for $c_3^\infty$ into the constraints, is equivalent to
  \[
   \begin{cases}
    c_1^\infty ( 1 + \kappa c_2^\infty)
     = a,
     \\
    c_2^\infty (1 + \kappa c_1^\infty)
     = b,
     \\
    c_3^\infty
     = \kappa c_1^\infty c_2^\infty
   \end{cases}
   \quad \Leftrightarrow \quad
   \begin{cases}
    c_1^\infty ( 1 + \frac{\kappa b}{1 + \kappa c_1^\infty})
     = a,
     \\
    c_2^\infty
     = \frac{b}{1 + \kappa c_2^\infty},
     \\
    c_3^\infty
     = \kappa c_1^\infty c_2^\infty
   \end{cases}
  \]
 The (uniquely non-negative) solution of this nonlinear system is then given by
  \begin{align*}
   c_1^\infty
    &= \frac{1}{2} \sqrt{\frac{2(a+b)}{\kappa} + (b-a)^2 + \kappa^{-2}} - \frac{1}{2} \left( \kappa^{-1} + (b-a) \right),
    \\
   c_2^\infty
    &= \frac{1}{2} \sqrt{\frac{2(a+b)}{\kappa} + (b-a)^2 + \kappa^{-2}} - \frac{1}{2} \left( \kappa^{-1} - (b-a) \right),
    \\
   c_3^\infty
    &= \frac{1}{2} \sqrt{\frac{2(a+b)}{\kappa} + (b-a)^2 + \kappa^{-2}} + \frac{1}{2} \left( \kappa^{-1} + a + b \right).
    \\
  \end{align*}
 This is the unique equilibrium compatible with the initial data $\vec c^\infty \geq \vec 0$ such that $a, b > 0$.
 Note that as $a \rightarrow 0$ and/or $b \rightarrow 0$, the corresponding equilibria converge to the equilibria for $a = 0$ and/or $b = 0$, resp., as well.
 
 \subsection{Some comments on more general reaction-diffusion-sorption systems}

 As the statements and proofs for local-in-time well-posedness and positivity of solutions have only been presented for the particular model problem \eqref{MP}, some comments are in place for possible generalisations to more general reaction mechanisms.
 
 \subsubsection{Reactions in the bulk phase}
 Bulk chemistry can be generally allowed for in the local-in-time well-posedness and positivity results. These reactions may be handled analogously to the situation with bulk chemistry as an additional semilinear term appearing on the right hand side of the bulk diffusion equation, and which typically has enough regularity (polynomial form) to be included in the contraction mapping principle argument, cf.\ the standard literature on semilinear reaction-diffusion-equations; see e.g.\ \cite{Pie10}.
 
 \subsubsection{General reaction schemes for the surface chemistry}
 For general surface chemistry with reactions of general type $\sum_i \alpha_i^{\Sigma,a} A_i^\Sigma \rightleftharpoons \sum_i \beta_i^{\Sigma,a} A_i^\Sigma$, following the linearisation scheme as for the model problem \eqref{MP}, one obtains a linearised version of the Dirichlet type boundary conditions $\kappa_a^f \vec c|_\Sigma^{\vec \alpha^{\Sigma,a}} = \kappa_a^b \vec c|_\Sigma^{\vec \beta^{\Sigma,a}}$ as
  \[
   \kappa_a^f \sum_{i: \alpha_i^{\Sigma,a} > 0} \alpha_i^{\Sigma,a} \vec c^{\vec \alpha^{\Sigma,a} - \vec e_i} - \kappa_a^b \sum_{i: \beta_i^{\Sigma,a} > 0} \beta_i^{\Sigma,a} \vec c^{\vec \beta^{\Sigma,a} - \vec e_i}
    = h_a
  \]
plus no-flux boundary conditions corresponding to conserved quantities. Very often, but depending on the particular structure of the reaction mechanisms, also for such more general reaction-diffusion-sorption systems, the strategy used in the proof of theorem \ref{thm:l-it_existence} may be employed to establish local-in-time existence of strong solutions, provided the initial data $\vec c^0$ are in $\BB_{pp}^{2-2/p}(\Omega;\R^N)$ and satisfy the compatibility conditions.
 However, this approach is, in general, quite tedious and essentially relies on the diagonal choice of the diffusion matrix $\bb D$ in the bulk, so that the reaction-diffusion equations are only coupled via the boundary conditions, and in the general case additionally over the semilinear reaction rate functions. In that case well-known maximal regularity and optimality results for the heat equation with inhomogeneous Neumann and Dirichlet boundary condition can be used again.
 For more general classes of diffusion in the bulk, say Fick--Onsager or Maxwell--Stefan diffusion, cross-diffusion in the bulk is allowed (and from a thermodynamic perspective even obligatory), so that for these types of systems a different approach is needed.
 For results in this direction, a possible approach is via the general maximal regularity and optimality results in the spirit of \cite{DeHiPr03} and \cite{DeHiPr07}, which need to be adjusted to the particular situation here, where at the same point $\vec z \in \Sigma$ both Neumann and Dirichlet type boundary conditions are imposed on the system; a situation which is not covered by the theory in \cite{DeHiPr03} and \cite{DeHiPr07} as the definition of the principle part of the boundary symbol used there leads to a boundary symbol without Dirichtlet type terms (and, hence, the Lopatinskii--Shapiro condition for the principle parts does not hold then). The abstract theory of \cite{DeHiPr03} and \cite{DeHiPr07}, however, can be extended by slightly adjusting the notion of the principal boundary symbol; for related results, see \cite{DeKa13}, where the notion of Newton polygons is heavily used.
 For the model problem \eqref{MP} and a large class of more general reaction-diffusion systems, e.g.\ reversible versions of the free radical addition, nucleophilic substitution, the Lindemann--Hinshelwood mechanism and electrocatalysis models, the latter approach is feasible as well, as one can check the Lopatinskii--Shapiro condition at least for non-negative initial data (which serve as reference data in the linearised system) to establish $\LL_p$-maximal regularity and optimality results.
 We abstain from giving more details here, but only refer to the upcoming paper \cite{AuBo19+2}.

 \section{Summary and outlook}
 This manuscript covers both the modelling and the analysis of a reaction-diffusion-sorption system arising from a bulk-surface reaction-diffusion-sorption system, when considering the formal limit of fast sorption processes and fast surface chemistry, and transmission takes place in a sublayer near the surface which is very thin.
 On the modelling side, several fast limit models have been proposed, based on a dimensional analysis of the thermodynamic processes. It has been demonstrated that depending on the time resolution, models of different accuracy and mathematical complexity arise. \newline
 The particular case of a fast-sorption--fast-surface-chemistry limit with fast transmission between bulk and surface has been further analysed in the second part of this manuscript.
 As the surface chemistry takes place even faster than the transmission between bulk and surface, the boundary conditions of the resulting reaction-diffusion-system are of \emph{mixed type}, in the sense that at the same time no-flux boundary conditions are imposed on the conserved quantities' part whereas the chemical reaction equilibrium on the surface corresponds to nonlinear relations between the area surface concentrations $c_i^\Sigma$, or their chemical potentials $\mu_i^\Sigma$.
 Due to the sorption equilibrium condition $s_i^\Sigma(\vec c|_\Sigma, \vec c^\Sigma) = 0$ corresponding to the fast-sorption limit, there further is a -- typically nonlinear -- relation between the are surface concentrations $c_i^\Sigma$ and the boundary traces of the bulk concentrations $c_i|_\Sigma$.
 Interpreting the sorption equilibrium condition rather as an equilibrium condition for the area and bulk chemical potentials, $\mu_i|_\Sigma = \mu_i^\Sigma$, one finds that the particular model for the surface chemical potentials does \emph{not} play a role for the resulting limit problem, as the reaction rates (in the reaction models used here) are completely determined by the surface chemical potentials $\mu_i^\Sigma$, so in the limit case by the traces of the bulk chemical potentials $\mu_i|_\Sigma$.
 The particular structure of the conservation vectors $\vec e^k$  for the surface chemical reactions being orthogonal to the stochiometric vectors $\vec \nu^{\Sigma,a}$ of the surface chemical reactions, can then be used to derive $\LL_p$-maximal regularity for a suitable linearisation of the semilinear fast limit model. From there (local-in-time) existence and uniqueness of strong solutions, blow-up criteria and a-priori bounds have been derived.

\end{document}